\documentclass[reqno]{amsart}

\usepackage[utf8]{inputenc}
\usepackage{amsmath, amsthm, amssymb, mathtools, array, enumitem, capt-of}
\usepackage{tikz}
\usetikzlibrary{knots}
\usetikzlibrary{decorations.pathmorphing}
\usetikzlibrary{decorations.pathreplacing}

\usepackage{xcolor}
\usepackage[colorlinks,
    linkcolor={red!50!black},
    citecolor={blue!50!black},
    urlcolor={blue!80!black}]{hyperref}
\usepackage{tikz}

\theoremstyle{plain}
\newtheorem{proposition}{Proposition}[section]
\newtheorem{theorem}[proposition]{Theorem}
\newtheorem{corollary}[proposition]{Corollary}
\newtheorem{lemma}[proposition]{Lemma}
\newtheorem*{theoremA}{Theorem A}

\newtheorem*{theoremB}{Theorem B}
\newtheorem*{theoremC}{Theorem C}
\newtheorem{theoremalph}{Theorem}

\theoremstyle{definition}
\newtheorem{definition}[proposition]{Definition}
\newtheorem{example}[proposition]{Example}

\theoremstyle{remark}
\newtheorem*{claim}{Claim}
\newtheorem{remark}[proposition]{Remark}

\newtheorem*{convention}{Convention}

\newcommand{\sgn}{\operatorname{sgn}}

\newcommand{\wt}{\widetilde}
\newcommand{\cl}{\operatorname{cl}}
\newcommand{\pt}{\operatorname{pt}}
\newcommand{\sm}{\setminus}

\newcommand{\Q}{\mathbb{Q}}

\newcommand{\Z}{\mathbb{Z}}
\newcommand{\C}{\mathbb{C}}
\newcommand{\R}{\mathbb{R}}
\newcommand{\CC}{\mathcal{C}}

\newcommand{\Id}{\operatorname{Id}}
\newcommand{\rk}{\operatorname{rk}}

\newcommand{\Diff}{\operatorname{Diff}}
\newcommand{\Top}{\operatorname{Top}}
\newcommand{\CAT}{\operatorname{CAT}}
\newcommand{\BD}{\operatorname{BD}}
\newcommand{\Wh}{\operatorname{Wh}}

\newcommand{\tmfrac}[2]{\mbox{\large$\frac{#1}{#2}$}} 

\begin{document}
\title{Smooth and topological almost concordance}

\author{Matthias Nagel}
\address{Mathematics \& Statistics,
McMaster University, Canada}
\email{nagel@cirget.ca}

\author[P. Orson]{Patrick Orson}
\address{Department of Mathematics, Boston College, USA}
\email{patrick.orson@bc.edu}

\author{JungHwan Park}
\address{Max Planck Institute for Mathematics, Vivatsgasse 7, 53111 Bonn, Germany}
\email{jp35@mpim-bonn.mpg.de}

\author{Mark Powell}
\address{D\'epartement de Math\'ematiques,
Universit\'e du Qu\'ebec \`a Montr\'eal, Canada}
\email{mark@cirget.ca}

\def\subjclassname{\textup{2010} Mathematics Subject Classification}
\expandafter\let\csname subjclassname@1991\endcsname=\subjclassname
\expandafter\let\csname subjclassname@2000\endcsname=\subjclassname

\begin{abstract}
We investigate the disparity between smooth and topological almost concordance of knots in general 3--manifolds $Y$. Almost concordance is defined by considering knots in $Y$ modulo concordance in~$Y \times [0,1]$ and the action of the concordance group of knots in~$S^3$ that ties in local knots. We prove that the trivial free homotopy class in every~$3$--manifold other than the $3$--sphere contains an infinite family of knots, all topologically concordant, but not smoothly almost concordant to one another.  Then, in every lens space and for every free homotopy class, we find a pair of topologically concordant but not smoothly almost concordant knots. Finally, as a topological counterpoint to these results, we show that in every lens space every free homotopy class contains infinitely many topological almost concordance classes.
\end{abstract}

\maketitle
\section{Introduction}
In this article we will exhibit and study an instance of the disparity between the smooth and topological categories in dimension 4 by looking at
concordance classes of knots in a general $3$--manifold, modulo the action of the concordance group of knots in $S^3$ that ties in local knots.

From now on $Y$ will always denote a closed, connected, oriented $3$--manifold, and concordance of knots in $Y$ will mean concordance in $Y\times[0,1]$. In both the smooth and topologically locally flat categories there is an action of the concordance group of knots in $S^3$ on the set of concordance classes of knots in~$Y$, given by local knotting. If two knots are in the same orbit of this action, we call them \emph{almost concordant} (see Section~\ref{sec:basics} for precise definition). In either category, all knots in~$S^3$ are almost concordant to the unknot. Two smoothly almost concordant knots are also topologically almost concordant. The goal of this article is to show that the converse statement is emphatically not the case. 

The main tool we use in this article to distinguish different smooth classes within a single topological class is the Ozsv{\'a}th-Szab{\'o} $\tau$ invariant for knots in $S^3$ \cite{MR2026543}. In order to leverage the topology of $Y$ and so manoeuvre into a situation where this invariant is applicable, we develop several new covering link arguments.  We also make essential use of a $3$--manifold embedding result due to the first author and H.\ Boden~\cite[Lemma 2.11]{Boden16}.

Our first main theorem shows that there is an infinite disparity between the categories in every manifold $Y\neq S^3$.

\begin{theoremalph}\label{thm:A}
For every $3$--manifold $Y \neq S^3$, there exists an infinite family of null homotopic knots $K_1, K_2, \dots$ in $Y$ that are all topologically concordant to one another, but are mutually distinct in smooth almost concordance.
\end{theoremalph}

The free homotopy class of a knot is preserved under both concordance and almost concordance. As every knot in the infinite family of Theorem~\ref{thm:A} is null homotopic, it is natural to wonder what effect the homotopy class has on the disparity between the categories. In this direction, our second main result shows that, at least for all lens spaces $Y$, every homotopy class of (unbased) maps $S^1 \to Y$ exhibits the disparity.

\begin{theoremalph}\label{thm:B}
  Let $Y = L(p,q)$, for $\gcd(p,q)=1$ and $p>1$, and let $x \in [S^1, Y]$ be any homotopy class.   Then there exist topologically concordant knots $K$ and $K'$ in~$Y$ representing the class $x$, that are distinct in smooth almost concordance.
\end{theoremalph}

A notable aspect of this theorem is the fact that the proof breaks up into four cases, each of which seems to require a different technique.  The first case, where $x$ is null homotopic, was covered by Theorem~\ref{thm:A}.  When~$Y\neq \R P^3$ and $x$ is homotopically essential, there are two cases: the case when~$x$ is order~$2$ and the case when the order is greater than~$2$. For these, different covering link arguments are needed, together with the $\tau$ invariant. Lastly, we give a completely distinct proof for the case of the nontrivial free homotopy class in $\mathbb{R}P^3$, using the $\Upsilon$ invariant of Ozsv{\'a}th-Stipsicz-Szab{\'o}~\cite{Ozsvath14}.

In the course of developing the arguments for proving Theorems~\ref{thm:A} and~\ref{thm:B}, we realised they could be adapted to work in an entirely topological context, in combination with the Levine-Tristram signatures, to prove the following theorem, the third main result of the article.

\begin{theoremalph}\label{thm:C}
  Let $Y = L(p,q)$, for $\gcd(p,q)=1$, $p>1$, and let $x \in [S^1, Y]$ be any homotopy class. Then there are infinitely many topological almost concordance classes representing the class $x$.
\end{theoremalph}

It was conjectured that for \emph{any} pair $(Y,x)$, equal
neither to $(S^3,e)$ nor to~$(S^1 \times S^2,[S^1 \times \pt])$, there are infinitely many topological almost concordance classes representing $x$~\cite[Conjecture 1.3]{FNOP}. Theorem~\ref{thm:C} confirms this conjecture for lens spaces.

\subsection*{Context for our results}
It was first observed independently in unpublished work of Casson and Akbulut
that there are knots which are topologically,
but not smoothly concordant; cf.~\cite[pp.\ 499, 502]{MR976591}.
Topological concordances exhibiting this phenomenon are always constructed using the topological embedding results of Freedman \cite{MR666159,Freedman-Quinn:1990-1,Garoufalidis-Teichner}. The earliest smooth concordance obstructions that do not obstruct topological concordance used the work of Donaldson \cite{MR710056}, but more modern arguments often obstruct smooth concordance using invariants from Heegaard-Floer theory, such as the $\tau$ invariant \cite{MR2026543}.

The naive method for transferring such examples from $S^3$ to any $3$--manifold does work. Precisely, take $K$ and $K'$ in $S^3$ that are topologically but not smoothly concordant in $S^3$, and locally embed them in any $3$--manifold $Y$. We use the aforementioned Boden-Nagel embedding result to
prove the following in Section~\ref{subsec:coveringlinks}.

\medskip

\noindent {\bf Proposition~\ref{prop:CtoCYinj}.}\,\,\,{\it In both the smooth and topologically locally flat category, if two local knots are concordant in a closed, oriented, connected $3$--manifold $Y$, then they are concordant in $S^3$.}
\medskip

From this it is straightforward to argue that $K, K' \subset Y$ are topologically but not smoothly concordant in $Y$ (see Corollary~\ref{cor:smoothtopgeneral} for details). However, the results in this paper do not follow in this naive way from previous examples. Almost concordance can be thought of as concordance in $Y$ modulo concordance in $S^3$, so it is precisely local examples of this sort which are invisible to the coarser relation of almost concordance. Hence the infinite families of examples we construct in this paper are a genuinely new instance of the smooth versus\ topological concordance phenomenon.

The study of concordance modulo local knotting for \emph{links} began with Milnor's~$\overline{\mu}$ invariants for links in $S^3$ \cite{Milnor:1957-1}. Versions of the $\overline{\mu}$ invariants for \emph{knots} in more general $3$--manifolds were constructed in \cite{MR1354382}, \cite{MR2012959} and \cite{Heck:2011}. The use of covering link methods to study concordance of knots in general 3--manifolds was explored in \cite{MR0458425}. The connections between invariants of covering links and $\overline{\mu}$ invariants were investigated in \cite{MR521732} and \cite{MR827299}.

More recently, specifically \emph{smooth} almost concordance invariants were developed in work of Celoria \cite{Celoria16}. Rolfsen \cite{zbMATH03917275} showed that smooth almost concordance is the same as PL $I$--equivalence of knots, so recent work of A.\ Levine \cite{Levine16} can also be viewed through the lens of smooth almost concordance. Topological almost concordance was studied by three of the present authors together with Friedl~\cite{FNOP}.
One of the main open questions left by these works was that of the disparity between the categories, which we resolve here. We note that while Celoria built an infinite family of knots, which he distinguished in smooth almost concordance by a version of the $\tau$ invariant in lens spaces, the techniques of \cite{FNOP} show that the knots in Celoria's family are moreover distinguished in topological almost concordance. Indeed the lifts of these knots from lens spaces~$L(p,1)$ to $S^3$ can be distinguished by their sets of pairwise linking numbers. Hence the families of examples we construct in the present article are the first of their kind.

We close the introduction by highlighting the following challenge: for a 3--manifold $Y$ and for $x \in [S^1,Y]$ homotopically essential, extend our Theorem~\ref{thm:B} to find infinitely many smooth almost concordance classes within a given topological almost concordance class representing~$x$.

\subsection*{Acknowledgements}

The authors thank Arunima Ray for several helpful conversations. We are also grateful to the anonymous referees for the detailed and thoughtful suggestions. JP thanks the Universit\'e du Qu\'ebec \`a Montr\'eal for its hospitality. MN and PO were supported by CIRGET postdoctoral fellowships. MP was supported by an NSERC Discovery grant.

\section{Basic constructions, covering links, and the \texorpdfstring{$\tau$}{tau} invariant}\label{sec:basics}

Let $Y$ be an oriented $3$--manifold, and fix an orientation on $S^1$. An embedding~$L\colon \bigsqcup_mS^1\hookrightarrow Y$, considered up to ambient isotopy, is called an \emph{$m$--component link} and a 1--component link is called a \emph{knot}. We will sometimes write $L\subset Y$ as a shorthand for a link. A knot in~$Y$ is called the \emph{unknot} if it bounds a $2$--disc in~$Y$. A knot~$K\subset Y$ is called \emph{local} if it lies inside a $3$--ball $D^3\subset Y$.
Given a knot in~$S^3$ construct a knot in~$Y$, by choosing a $3$--ball in $S^3$ that contains the knot, and then embedding this $3$--ball into $Y$ via an orientation preserving embedding.
We next show that this inclusion of a knot in $S^3$ as local knot in $Y$ gives a one-to-one correspondence  (up to respective ambient isotopy).

\begin{lemma}
Let $Y$ be an orientable $3$--manifold, and let $K$ and $J$ be two
knots contained in a $3$-ball~$K,J \subset D^3 \subset Y$.
Then the knots~$K$, $J$ are isotopic in $Y$ if and only if they are isotopic in
the $3$--ball~$D^3$.
\end{lemma}

\begin{proof}
First, we show that if $K$ and $J$ are isotopic in $Y$ then they are
also isotopic in $D^3$.
Denote by $\rho_t \colon Y \to Y$ an ambient isotopy from $K$ to $J$ in $Y$.
By scaling in~$\rho_1(D^3)$, we may assume that $\rho_1(D^3)  \subset D^3$,
and $\rho_1(K)$ is isotopic to $J$ in~$D^3$. Refine~$\rho_t$ even further:
by the Schoenflies theorem~$D^3 \setminus \operatorname{Int} \rho_1(D^3)$ is an annulus~$S^2 \times I$,
and thus arrange that~$\rho_1(D^3) = D^3$, while still
preserving that $\rho_1(K)$ is isotopic to $J$ in $D^3$. The isotopy
between $\rho_1(K)$ and $J$ can be extended to an ambient isotopy that
is supported entirely in the interior of $D^3$. Thus change~$\rho_t$ one
last time to achieve that $\rho_1(D^3) = D^3$,
and $\rho_1(K) = J$.

Forget about the ambient manifold~$Y$, and consider
the restriction~$\rho_1 \colon D^3 \to D^3$. By assumption this
is an orientation preserving diffeomorphism~$\rho_1 \in \operatorname{Diff}^+(D^3)$.
What is left  to prove is that~$K$ is isotopic to $\rho_1(K) = J$ in $D^3$, which follows from
the fact that the space~$\operatorname{Diff}^+(D^3)$ is path-connected~\cite{Cerf68}.

The other direction is immediate as an isotopy between $J$ and $K$ in $D^3$ embeds to an isotopy in $Y$.
\end{proof}

For a knot $J\subset S^3$ and~$n\in\Z_{>0}$,
we introduce the notation\[nJ:=\underbrace{J\#\dots\# J}_{n}.\]

For either $\CAT=\Diff$ or $\Top$, we define two links $L_0, L_1\subset Y$ to be \emph{$\CAT$--concordant} if
\begin{enumerate}[leftmargin=*]
  \item $\CAT=\Diff$: they are \emph{smoothly concordant}, that is, if there is a smooth embedding $\bigsqcup_mS^1 \times I \hookrightarrow Y \times I$ with the image of $\bigsqcup_mS^1 \times \{i\}$  in  $Y \times \{i\}$ and equal to $L_i$, for $i=0,1$.
     \item $\CAT = \Top$: they are \emph{topologically concordant}, that is, if there is a locally flat embedding $\bigsqcup_mS^1 \times I \hookrightarrow Y \times I$ with the image of $\bigsqcup_mS^1 \times \{i\}$  in  $Y \times \{i\}$ and equal to $L_i$, for $i=0,1$.
\end{enumerate}

Write $[U,V]$ for the set of free (i.e.\ unbased) homotopy classes of maps $U\to V$ between path connected topological spaces $U,V$, and write $e\in[U,V]$ for the class of the constant map. For a $3$--manifold $Y$ and a class $x \in [S^1,Y]$, define the concordance set $\CC^{\CAT}_x(Y)$ to be the set of knots $S^1 \subset Y$ representing the class $x$, up to $\CAT$--concordance. Write $\CC^{\CAT}:=\CC^{\CAT}_e(S^3)$ for the concordance group of knots in $S^3$. The connected sum of pairs $(S^3,J)\# (Y,K)$ defines a new knot $(Y,J\#K)$, which is freely homotopic to $K$ in~$Y$, since all knots in $S^3$ are freely null homotopic.

\begin{definition}\label{def:local}
For either $\CAT=\Diff$ or $\Top$, and for each pair $(Y,x)$, the \emph{action of local knots} of the concordance group~$\mathcal{C}^{\CAT}$
on the set~$\mathcal{C}^{\CAT}_x(Y)$ is
defined by
\[\mathcal{C}^{\CAT}\times\mathcal{C}^{\CAT}_x(Y)\to \mathcal{C}^{\CAT}_x(Y)\qquad
([J],[K])\mapsto [J\# K].\] We say that two knots in $\CC_x^{\Diff}(Y)$ are \emph{smoothly almost concordant} if they lie in the same orbit of the action of local knots, and we say that two knots in $\CC_x^{\Top}(Y)$ are \emph{topologically almost concordant} if they lie in the same orbit of this action of local knots.
\end{definition}

\begin{convention} Throughout this paper, in link diagrams, whenever a tangle is drawn from left to right of the page, it is implicit that the diagram should be understood as the closure of that tangle by the trivial braid.
\end{convention}

\subsection{Basic constructions}\label{subsec:lensspaces}

Recall the following construction of the lens space~$L(p,q)$ for $p > 1$ and
$1 \leq q < p$ coprime to $p$. Write
\[ S^3 = \Big\{ (z_1, z_2) \in \C^2 : |z_1|^2 + |z_2|^2 = 1 \Big\}. \]
The solid torus~$V = \{ (z_1, z_2) \in S^3 \mid |z_2|^2 \leq \frac{1}{2} \}$
and its complement give
a genus~$1$ Heegaard splitting of $S^3$. Define the map
\[ f\colon S^3\to S^3;\quad (z_1, z_2) \mapsto \big(z_1 \exp(2\pi i /p), z_2 \exp(2\pi i q/p)\big).\]
The $3$--manifold~$L(p,q)$ is obtained as the quotient space $S^3/\sim$ using the equivalence relation $z\sim f(z)$. Note that $f(V)=V$, and so we can consider the quotient~$V/\sim$.	
The boundary~$\partial V = S^1 \times S^1$
inherits its product structure from $\C^2$. The second factor bounds a disc in $V$ and
the first factor bounds in the complement~$V^c$ of $V$. Identify
$V$ with $S^1 \times D^2$ by sending $(z_1, z_2)\in V$ to $(z_1 / |z_1|, z_2\sqrt{2}  \big)$.
Under this identification, the quotient map $g\colon V\to V/\sim$ is given by
\[g \colon V=S^1 \times D^2 \to S^1 \times D^2;\quad (t,z) \mapsto (t^p , t^{-q} z).\]
We see that the boundary of a meridional disc of $V^c$ is mapped to a curve in
$\partial \left(V/\sim\right) = S^1\times S^1$ that winds $p$ times around the first $S^1$ factor and
$-q$ times around the second $S^1$ factor; in other words, to a curve of slope~$-p/q$.
The orientations are inherited naturally from $\C$, explaining the minus sign on $q$.

\begin{figure}[h]

\begin{tikzpicture}[scale=0.95]

\begin{scope}[shift={(2.5,3.7)}]

\begin{scope}[xscale=0.25, yscale=0.25, shift={(25.5,3.75)}]
\draw[thick] (-0.5, 0.5) -- (0,0);
\draw[thick] (0.5, 0.5) -- (0,0);
\end{scope}

\begin{knot}[
clip radius=7pt,
clip width=5,
]

\strand[black,thick] (2,2) ..  controls +(0,0) and +(-0,0) .. (7,2);

\strand[black,thick] (2,1.75) ..  controls +(0,0) and +(-0,0) .. (7,1.75);

\strand[black,thick] (2,1.5) ..  controls +(0,0) and +(-0,0) .. (7,1.5);

\strand[black,thick] (2,0.25) ..  controls +(0,0) and +(-0,0) .. (7,0.25);

\strand[black,thick] (2,0) ..  controls +(0,0) and +(-0,0) ..  (7,0);

\strand[black,thick] (6,2.5) .. controls +(0.5,0) and +(0.5,0) .. (6,-0.5) .. controls +(-0.5,0) and +(-0.5,0) .. (6,2.5);

\flipcrossings{1,3,5,7,9}

\end{knot}

\node at (6.4,-0.7) {$-p/q$};
\node at (2.1,1) {$\vdots$};
\node at (6.9,1) {$\vdots$};

\draw [thick, draw=black, fill=white](3.7,2.2)
to [out=right, in=right, looseness=1.5] (3.7,-0.2)
to [out=left, in=left, looseness=1.5] (3.7,2.2);
\node[rotate=0] at (3.7,1) {$P$};

\draw [decorate,decoration={brace,amplitude=6pt,raise=8pt}] (1.8,-0.1) -- (1.8,2.1);
\node at (0.95,1) {$\ell$};

\end{scope}

\begin{scope}[shift={(0,0)}]\begin{knot}[
clip radius=7pt,
clip width=5,
]

\strand[black,thick] (2,2) ..  controls +(0,0) and +(-0,0) .. (7.5,2);

\strand[black,thick] (2,1.75) ..  controls +(0,0) and +(-0,0) .. (7.5,1.75);

\strand[black,thick] (2,1.5) ..  controls +(0,0) and +(-0,0) .. (7.5,1.5);

\strand[black,thick] (2,0.25) ..  controls +(0,0) and +(-0,0) .. (7.5,0.25);

\strand[black,thick] (2,0) ..  controls +(0,0) and +(-0,0) ..  (7.5,0);

\flipcrossings{13,11,9,7,5}

\end{knot}

\node at (2.1,1) {$\vdots$};
\node at (7.95,0.9) {$\cdots$};

\draw [thick, draw=black, fill=white](3.5,2.2)
to [out=right, in=right, looseness=1.5] (3.5,-0.2)
to [out=left, in=left, looseness=1.5] (3.5,2.2);
\node[rotate=0] at (3.5,1) {$P$};

\draw [thick, draw=black, fill=white](6,2.2)
to [out=right, in=right, looseness=1.5] (6,-0.2)
to [out=left, in=left, looseness=1.5] (6,2.2);
\node[rotate=0] at (6,1) {$P$};

\draw [decorate,decoration={brace,amplitude=6pt,raise=8pt}] (1.8,-0.1) -- (1.8,2.1);
\node at (0.95,1) {$\ell$};

\end{scope}

\begin{scope}[shift={(6.1,0)}]\begin{knot}[
clip radius=7pt,
clip width=5,
]

\strand[black,thick] (2.2,2) ..  controls +(0,0) and +(-0,0) .. (7,2);

\strand[black,thick] (2.2,1.75) ..  controls +(0,0) and +(-0,0) .. (7,1.75);

\strand[black,thick] (2.2,1.5) ..  controls +(0,0) and +(-0,0) .. (7,1.5);

\strand[black,thick] (2.2,0.25) ..  controls +(0,0) and +(-0,0) .. (7,0.25);

\strand[black,thick] (2.2,0) ..  controls +(0,0) and +(-0,0) ..  (7,0);

\flipcrossings{}

\end{knot}

\node at (6.9,1) {$\vdots$};

\draw [thick, draw=black, fill=white](3.7,2.2)
to [out=right, in=right, looseness=1.5] (3.7,-0.2)
to [out=left, in=left, looseness=1.5] (3.7,2.2);
\node[rotate=0] at (3.7,1) {$P$};

\draw [thick, draw=black, fill=white]
       (5.5, -0.2) -- (6.5,-0.2) -- (6.5,2.2) -- (5.5,2.2) -- cycle;
\node[rotate=0] at (6,1) {$q$};

\end{scope}

\draw [decorate,decoration={brace,amplitude=6pt,raise=8pt}]  (10.8,-0.2) -- (2.5,-0.2);
\node at (6.65,-1.1) {$p$};

\end{tikzpicture}
\caption{A general link~$L$ in~$L(p,q)$ represented using a tangle~$P$, as indicated. Below is its covering link~$\wt L$ in the universal cover~$S^3$. The $q$--box denotes full twists involving all strands; the twists are right-handed with respect to the left-to-right direction.}
\label{fig:CoveredLink}
\end{figure}
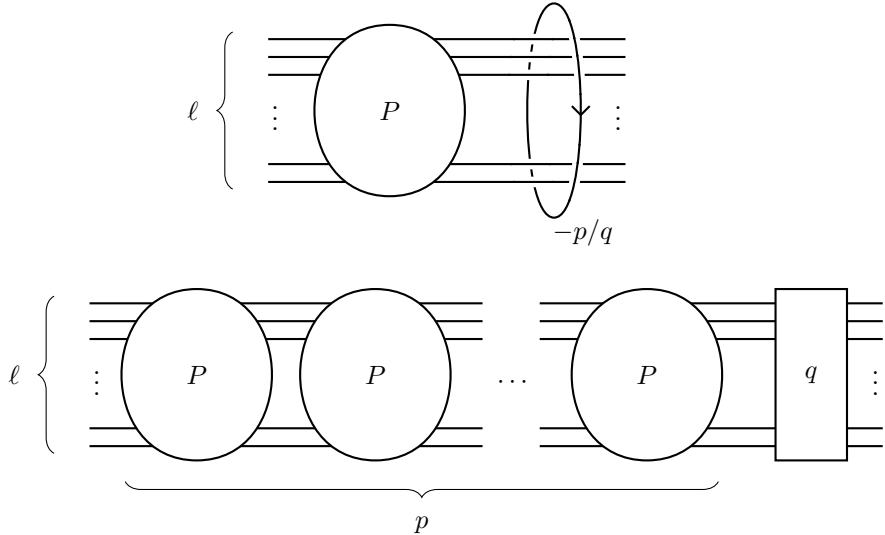

\begin{lemma}\label{lem:coveringlink-lens-spaces}
Let $L\subset L(p,q)$ and $\wt L\subset S^3$ be the links depicted in Figure~\ref{fig:CoveredLink}. Then the link~$\wt L$ is the covering link of $L$ in the universal cover $S^3$. If $L$ is moreover a knot, with each strand in Figure~\ref{fig:CoveredLink} oriented from left to right, then $\wt L$ is a $\gcd(\ell,p)$--component link.
\end{lemma}

\begin{proof}
The link $L$ is contained in the solid torus $U:=S^1\times D^2\subset L(p,q)$, which is the exterior of a neighbourhood of the $-p/q$ surgery curve. The covering map~$g \colon V \to U$, from above, is the composition of the maps
\begin{alignat*}{2}
g \colon S^1 \times D^2 &\xrightarrow{\phi}& S^1 \times D^2 &\to S^1 \times D^2\\
(t,z) &\mapsto &(t, t^{-q} z)&\\
&&(t,z) &\mapsto (t^p, z).
\end{alignat*}
The map~$\phi$ is a diffeomorphism with inverse $\phi^{-1}(t,z) = (t, t^{q}z)$.
The preimage of the second map is obtained by taking the tangle $P$ and stacking it $p$ times.

The second statement is obtained from the first by a counting argument. By viewing the tangle $P$ as a permutation on the strands as we pass through it from left to right in the diagram, define an element $\alpha_L\in S_{\ell}$ the symmetric group on~$\ell$ letters. Because $L$ is now assumed to be knot, $\alpha_L$ is an $\ell$--cycle as it has order $\ell$. Using the first part of the lemma, and noting that the $q$--box does not permute the strands, we must now count how many disjoint cycles there are in $(\alpha_L)^p$. But it is an elementary exercise in modular arithmetic that a length $\ell$ cycle raised to the power~$p$ decomposes into $\gcd(\ell, p)$ disjoint cycles, each of length~$\ell/\gcd(\ell, p)$.
\end{proof}

\noindent We will frequently use the following knots and links.

\begin{definition}For a knot $J\subset S^3$ and $r,s\in \Z$, the \emph{$(r,s)$ cable of $J$} is the link~$J_{r,s}$ in $S^3$ on the boundary of the $0$--framed closed tubular neighbourhood of $J$ that winds $r$ times around the longitudinal direction and $s$ times around the meridional direction. The \emph{$(r,s)$ torus link} $T_{r,s}$ in $S^3$ is the $(r,s)$ cable of the unknot.
\end{definition}

Note that for any $J\subset S^3$ and any $r,s\in \Z$, the number of components of $J_{r,s}$ is precisely $d:=\gcd(r,s)$, and each component knot is a copy of $J_{\frac{r}{d},\frac{s}{d}}$. The Alexander polynomial of the torus knot $T_{r,s}$ is given by~\cite[p. 119]{Lickorish97}
\begin{equation}\label{eq:1}\Delta_{T_{r,s}}(t)=\frac{(t^{rs}-1)(t-1)}{(t^r-1)(t^s-1)}.\end{equation}
For a knot $J\subset S^3$, the \emph{Seifert genus} $g_3(J)$ denotes the minimal integer $g$ for which~$J$ bounds an embedded oriented genus $g$ surface in $S^3$. For $r,s>0$ coprime, the Seifert genus of the $(r,s)$ torus knot is $(r-1)(s-1)/2$, which follows from
Seifert's inequality~\cite[Proposition 6.13]{Lickorish97}.

\begin{definition}\label{def:Jbox}
For any $n>0$ and a knot $J\subset S^3$, the \emph{$J$--box} is the tangle diagram such that the diagram in Figure \ref{fig:Jbox} is the $(n,0)$ cable of $J$, where the strands are oriented from left to right.
\begin{figure}[h]

\begin{tikzpicture}[scale=.6]

\begin{scope}[shift={(2,3.7)}]\begin{knot}[
clip radius=7pt,
clip width=5,
]

\strand[black,thick] (2,2) ..  controls +(0,0) and +(-0,0) .. (7,2) ;

\strand[black,thick] (2,1.75) ..  controls +(0,0) and +(-0,0) .. (7,1.75) ;

\strand[black,thick] (2,1.5) ..  controls +(0,0) and +(-0,0) .. (7,1.5) ;

\strand[black,thick] (2,0.25) ..  controls +(0,0) and +(-0,0) .. (7,0.25) ;

\strand[black,thick] (2,0) ..  controls +(0,0) and +(-0,0) ..  (7,0);

\flipcrossings{}

\end{knot}

\node at (2.1,1) {$\vdots$};
\node at (6.9,1) {$\vdots$};

\draw [thick, draw=black, fill=white]
       (3.5, -0.2) -- (5.5,-0.2) -- (5.5,2.2) -- (3.5,2.2) -- cycle;
\node[rotate=0] at (4.5,1) {$J$};

\draw [decorate,decoration={brace,amplitude=6pt,raise=8pt}] (1.8,-0.1) -- (1.8,2.1);
\node at (0.6,1) {$n$};

\end{scope}

\end{tikzpicture}
\caption{The $J$--box is the tangle such that this diagram is the $(n,0)$ cable of $J$, where the strands are oriented from left to right.}
\label{fig:Jbox}
\end{figure}
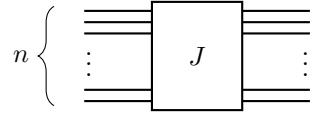
\end{definition}

\begin{definition}Given a knot $J\subset S^3$, the \emph{Bing double} $\BD(J)$ is the 2--component link in $S^3$ depicted in Figure~\ref{fig:BDWh}~(a). The \emph{positive Whitehead double} $\Wh^+(J)$ is the knot depicted in Figure~\ref{fig:BDWh}~(b). The \emph{negative Whitehead double} $\Wh^-(J)$ is the knot depicted in Figure~\ref{fig:BDWh}~(c).
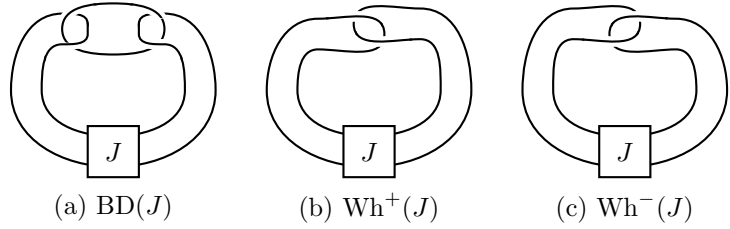
\begin{figure}[h]
\begin{tikzpicture}[scale=0.68]
\begin{knot}[
clip width=5,
]

\strand [black, thick] (0,0)
to [out=right, in=down] (2,1.5)
to [out=up, in=right] (1,3)
to [out=left, in=up] (0.5,2.8)
to [out=down, in=up] (0.5,2.6)
to [out=down, in=left] (1,2.4)
to [out=right, in=up] (1.4,1.5)
to [out=down, in=right] (0,0.6)
to [out=left, in=down] (-1.4,1.5)
to [out=up, in=left] (-1,2.4)
to [out=right, in=down] (-0.5,2.6)
to [out=up, in=down] (-0.5,2.8)
to [out=up, in=right] (-1,3)
to [out=left, in=up] (-2,1.5)
to [out=down, in=left] (0,0)
;

\strand [black, thick] (0,3.2)
to [out=right, in=up] (1,2.7)
to [out=down, in=right] (0,2.2)
to [out=left, in=down] (-1,2.7)
to [out=up, in=left] (0,3.2)
;

\flipcrossings{1,4}
\end{knot}


\begin{scope}[shift={(-0.5,-0.2)}]
\draw [thick, draw=black, fill=white]
       (0,0) -- (0,1) -- (1,1) -- (1,0) -- cycle;
\node at (0.5,0.5) {$J$};
\end{scope}

\begin{scope}[shift={(5,0)}]\begin{knot}[
clip width=5,
]

\strand [black, thick] (0,0)
to [out=right, in=down] (2,1.5)
to [out=up, in=right] (1,3.1)
to [out=left, in=up] (-0.3,2.8)
to [out=down, in=left] (1,2.5)
to [out=right, in=up] (1.4,1.5)
to [out=down, in=right] (0,0.6);
\strand [black,thick] (0,0.6)
to [out=left, in=down] (-1.4,1.5)
to [out=up, in=left] (-1,2.3)
to [out=right, in=down] (0.3,2.6)
to [out=up, in=right] (-1,2.9)
to [out=left, in=up] (-2,1.5)
to [out=down, in=left] (0,0)
;

\flipcrossings{1}
\end{knot}
\end{scope}

\begin{scope}[shift={(4.5,-0.2)}]
\draw [thick, draw=black, fill=white]
       (0,0) -- (0,1) -- (1,1) -- (1,0) -- cycle;
       \node at (0.5,0.5) {$J$};
\end{scope}

\begin{scope}[shift={(10,0)}]\begin{knot}[
clip width=5,
]

\strand [black, thick] (0,0)
to [out=right, in=down] (2,1.5)
to [out=up, in=right] (1,2.9)
to [out=left, in=up] (-0.3,2.6)
to [out=down, in=left] (1,2.3)
to [out=right, in=up] (1.4,1.5)
to [out=down, in=right] (0,0.6);
\strand [black,thick] (0,0.6)
to [out=left, in=down] (-1.4,1.5)
to [out=up, in=left] (-1,2.5)
to [out=right, in=down] (0.3,2.8)
to [out=up, in=right] (-1,3.1)
to [out=left, in=up] (-2,1.5)
to [out=down, in=left] (0,0)
;

\flipcrossings{2}
\end{knot}
\end{scope}

\begin{scope}[shift={(9.5,-0.2)}]
\draw [thick, draw=black, fill=white]
       (0,0) -- (0,1) -- (1,1) -- (1,0) -- cycle;
       \node at (0.5,0.5) {$J$};
\end{scope}

\node at (0,-0.8) {(a) $\BD(J)$};
\node at (5,-0.8) {(b) $\Wh^+(J)$};
\node at (10,-0.8) {(c) $\Wh^-(J)$};
%
%
%
%
%
%
%
%
%
%
%
%
%

\end{tikzpicture}

\caption{The Bing and Whitehead doubles of a knot $J\subset S^3$. These are invariant under orientation reversals.}
\label{fig:BDWh}
\end{figure}
\end{definition}

For $J\subset S^3$, both $\Wh^+(J)$ and $\Wh^-(J)$ have Alexander polynomial 1, so by~\cite{MR666159} both of these knots are topologically concordant to the unknot $U\subset S^3$. Note as well that $\Wh^{-}(-J)=-\Wh^+(J)$, where $-J$ is a mirror image of $J$ with the orientation reversed.

\subsection{Covering links}\label{subsec:coveringlinks}

We will discuss some general behaviour of knots and links under covering maps. Let $\pi\colon \wt{Y}\to Y$ be any covering (not necessarily of finite degree).

Let $J\subset D^3\subset Y$ be a local knot. The connected components $\wt{D}$ of the preimage $\pi^{-1}(D^3)$ each contain a unique connected component $\wt{J}$ of $L$. Moreover, the restriction~$\pi\colon (\wt{D},\wt{J})\to (D^3, J)$ is a homeomorphism. Recall the local knots in a given $3$--manifold are in natural one-to-one correspondence with knots in $S^3$. In particular we have shown that the local knots $J\subset Y$ and $\wt{J}\subset\wt{Y}$ correspond to the same knot in $S^3$ and that there are as many connected components of $\pi^{-1}(J)$ as the degree of the cover $\pi$.

Let $K\subset Y$ be a knot and write $L:=\pi^{-1}(K)$ for the preimage of~$K$. Suppose for now that $K$ is null homotopic. Then $K$ is homotopic to a local knot~$J\subset D^3\subset Y$. Combining the observations of the previous paragraph with the homotopy lifting property of $\pi$, we must have also that for each connected component~$\wt{K}$ of $L$, the cover~$\pi\colon \wt{K}\to K$ has degree~$1$. So $L$ has as many connected components as the degree of the cover~$\pi$.

Now drop the assumption, in this discussion, that $K$ be null homotopic. But assume for the rest of this discussion that for each connected component $\wt{K}\subset L$, the induced cover $\pi\colon \wt{K}\to K$ is of finite degree.
For $\wt{K}\subset L$ a connected component, write $d(\wt{K})$ for the degree of the induced cover~$\pi\colon \wt{K}\to K$. Let $J$ be a knot in~$S^3$. If we tie $J$ into $K$, the component of $L$ that was $\wt{K}$ becomes the connected component~$\wt{K}\#d(\wt{K}) J\subset \pi^{-1}(K\# J)$. In particular, the connected components of $\pi^{-1}(K\# J)$ and of $L$ are in one-to-one correspondence.

Fix either $\CAT=\Diff$ or $ \Top$ and let $K'\subset Y$ be almost concordant to $K$, with preimage $\pi^{-1}(K')=:L'$. By definition, there exists some $J\subset S^3$ such that there is a concordance from $K\#J$ to $K'$ in $Y$; write this concordance as $A=S^1 \times I \subset Y \times I$.
The preimage $\wt{A}$ of $A$ under the covering $\pi\times\Id\colon \wt{Y}\times I\to Y\times I$ is such that $\wt{A}\cap(\wt{Y}\times\{0\})=\pi^{-1}(K\# J)$, and $\wt{A}\cap(\wt{Y}\times\{1\})=L'$. Write $\bigcup_i\wt{A}_i\subset \wt{Y}\times I$, where~$\wt{A}_i$ denotes a connected component. Each $\wt{A}_i$ is a concordance (as opposed to some other surface), as can be seen by considering that the fundamental group~$\pi_1(\wt{A}_i)$ is isomorphic to an index~$d(\wt{K})$ subgroup of~$\pi_1(A_i)\cong\Z$, and is thus isomorphic to~$\Z$. Define a bijection between the set of components of $\wt{A}$ and the set of components of $\pi^{-1}(K\#J)$ by associating to each component~$\wt A_i$ the connected component of~$\pi^{-1}(K\#J)$ appearing on its negative boundary. Similarly, we define a bijection between the set of components of $\wt{A}$ and of $L'$ using the positive boundaries. Note that the components of $\pi^{-1}(K\#J)$ are also in bijective correspondence with
the components of $L$, and so also the set of components of $L$ and $L'$ have the
same cardinality.

We recall a result due to the first author and H.\ Boden~\cite[Lemma 2.11]{Boden16}.

\begin{lemma}[Boden-Nagel]\label{lem:boden-nagel}
  Let $Y$ be a 3--manifold and write $N := \wt{\cl(Y \sm D^3)}$ for the universal cover of $Y$ with a 3--disc removed.  Then there is a smooth embedding~$\phi\colon N \hookrightarrow S^3$.
\end{lemma}

\begin{remark}
The fact that the embedding $\phi$ is smooth was not discussed in \cite[Lemma 2.11]{Boden16}, so we discuss it here.
This is clear for prime manifolds by the construction of \cite[Lemma 2.10]{Boden16}.
For composite 3--manifolds, the problem of smoothing reduces to the following:
given two $3$--manifolds~$M$ and $N$ with smooth embeddings~$\phi_M \colon M \to S^3$,
$\phi_N \colon N \to S^3$ such that the images of these embeddings
have exactly a $2$--sphere in common, we must arrange the smooth
structure on $M \cup_{S^2} N$ so that $\phi_M \cup_{S^2} \phi_N$ is a
smooth map.
Note that $M$ and $N$ are glued along a smooth map.
Furthermore, a normal bundle of $S^2$ gives rise to a collar in
$M$, and in $N$. Use these collars to construct a smooth structure
on $M \cup_{S^2} N$~\cite[Section~13.8]{Janich82} and the resulting
map~$\phi_M \cup_{S^2} \phi_N$ will indeed be smooth.
\end{remark}

This embedding result will be used in this article to transport a concordance in a general $3$--manifold into $S^3$. For example, this embedding result can be used to prove the following proposition.

\begin{proposition}\label{prop:CtoCYinj}Let $\CAT=\Diff$ or $\Top$. For a closed, connected, oriented~$3$--manifold $Y$, the natural map $\mathcal{C}^{\CAT}\to \mathcal{C}^{\CAT}_e(Y)$, defined by embedding a knot $K\subset S^3$ as a local knot in $Y$, is injective.
\end{proposition}

\begin{proof}
Denote the natural map from the set of knots in $S^3$ to the set of knots in~$Y$, given by sending a knot $K\subset S^3$ to a local knot in $Y$, by $\iota$. Suppose $J, J'\subset S^3$ are such that $\iota(J)$ is concordant in $Y$ via an annulus $A\subset Y\times[0,1]$ to $\iota(J')$. There is a $3$--disc~$D^3\subset Y$ disjoint from the projection of $A$ to $Y$. To see this, take an arc running parallel to $A$, straighten it via isotopy in the complement of $A$ to be $\pt\times[0,1]$, then thicken the point to a $3$--disc. We excise this $3$--disc and write $N:=\cl(\wt{Y\sm D^3})$ for the closure of the universal cover. Apply Lemma~\ref{lem:boden-nagel} to obtain an embedding $\phi\colon N\hookrightarrow S^3$. By the discussion above, any connected component $\wt{A}\subset N\times[0,1]$ of the preimage of $A$ under the covering map is a concordance in $N$ from a connected component $\wt{J}$ of the lift of $\iota(J)$ to a connected component $\wt{J'}$ of the lift of $\iota(J')$. The image of $\wt{A}$ under $\phi\times\Id\colon N\times[0,1]\to S^3\times[0,1]$ is a concordance from $\phi(\wt{J})$ to $\phi(\wt{J'})$. But by the discussion above, each connected component of the preimage of a local knot is that same local knot. So we obtain a concordance between $J$ and  $J'$ in $S^3\times [0,1]$, as required.
\end{proof}

\begin{remark}
In the special case that $Y=S^1\times S^2$, this result was later obtained in \cite[Theorem 2.5]{DNPR} using a different method. While the method in that paper does not generalise to arbitrary $3$--manifolds, it does extend from annuli to any type of surface cobordism.
\end{remark}

\begin{corollary}\label{cor:smoothtopgeneral} In any closed, connected, oriented $3$--manifold $Y$ there is an infinite family of knots $K_1, K_2, \ldots$ topologically concordant to the unknot but mutually distinct in smooth concordance.
\end{corollary}

\begin{proof}
For $S^3$ such families are known, and we construct one in Example~\ref{exm:SmoothFamily}.
Take any such family in $S^3$ and embed it locally in $Y$.
Suppose $K_i,K_j\subset Y$ are knots in this family, with $i\neq j$. The knots $K_i$ and $K_j$ are topologically concordant in $Y$, as the topological concordance from $K_i$ to $K_j$ in $S^3$ can be locally embedded in~$D^3\times[0,1]\subset Y\times[0,1]$. The existence of a smooth concordance from~$K_i$ to~$K_j$ is precluded by Proposition~\ref{prop:CtoCYinj} as~$K_i$ is not smoothly concordant to~$K_j$ in~$S^3$.
\end{proof}

\noindent Here is an observation that will be crucial in this article.
\begin{proposition}\label{prop:almostimpliesconc}
Fix either $\CAT=\Diff$ or $ \Top$ and let $K,K'\subset Y$ be almost concordant. Let $\pi\colon \wt{Y}\to Y$ be a cover such that each connected component of both~$\pi^{-1}(K)=:L$ and $\pi^{-1}(K')=:L'$ is unknotted in $\wt{Y}$. Then $L$ is concordant to~$L'$.
\end{proposition}

\begin{proof}
For some $J\subset S^3$, there is a concordance $A$ from $K\#J$ to $K'$. Choose a connected component $\wt{A}$ of the preimage $\pi^{-1}(A)$. By the discussion earlier in this section, $\wt{A}$ is a concordance from $\wt{K}\# d(\wt{K})J$ to $\wt{K'}$, where $\wt{K}$ and $\wt{K'}$ are connected components of $L$ and of $L'$ respectively. As $\wt{K}$ and $\wt{K'}$ are unknots, this is a concordance in $\wt{Y}$ from $d(\wt{K})J$ to $U$.

\begin{claim}
The knot in $S^3$ corresponding to the local knot $d(\wt{K})J\subset \wt{Y}$ is concordant to the unknot in $S^3$.
\end{claim}

(We prove this claim using an argument extremely similar to the proof of Proposition~\ref{prop:CtoCYinj}. The reason this proposition cannot simply be applied directly is that, in the current situation, the ambient 3--manifold $\wt{Y}$ for the local knots $d(\wt{K})J$ and $U$ is not necessarily compact.)

Take $D^3\subset Y$ disjoint from the projection of $A$ to $Y$. The concordance $\wt{A}$ in the cover $\wt{Y}$ is disjoint from the preimage of $D^3\times[0,1]$ under $\pi$. Hence $\wt{A}$ lifts to the universal cover of $Y\sm D^3$. Use Lemma~\ref{lem:boden-nagel} to embed the closure of the universal cover of $Y\sm D^3$ into~$S^3$. We obtain a lift of the concordance $\wt{A}$, now embedded in $S^3\times [0,1]$. As both~$d(\wt{K})J$ and $U$ are local knots, the components of their lifts to the universal cover are again copies of $d(\wt{K})J$ and of $U$ respectively. So the null concordance of $d(\wt{K})J$ in $\wt{Y}$ implies moreover that $d(\wt{K})J$ is null concordant in $S^3 \times [0,1]$, as required.

Consequently, each component $\wt{K}\#d(\wt{K})J$ has a concordance to $\wt{K}$ which is the product concordance outside a $3$--ball.
Because of that, we can collect them into a single simultaneous concordance~$A'$ from $\pi^{-1}(K\# J)$ to $L$ in $\wt{Y} \times [0,1]$.
We note that if the concordances had not been local, we would not know whether the individual concordances could be simultaneously used to construct an embedded~$A'$. Concatenating the concordances $\pi^{-1}(A)$ and $A'$ we obtain the required concordance from $L'$ to $L$.
\end{proof}

\subsection{Rational concordance and the \texorpdfstring{$\tau$}{tau} invariant}

Two knots $J_1, J_2\subset S^3$ are called \emph{rationally smoothly concordant} if they are smoothly concordant in a rational homology $S^3\times [0,1]$, that is if there is a manifold with boundary $(X,S^3\times\{0,1\})$ and a smooth embedding $S^1 \times I \hookrightarrow  X$ with the image of $S^1 \times \{i\}$  in  $S^3 \times \{i\}$ equal to $J_i$, for $i=0,1$, and $H_*(X;\Q)=H_*(S^3;\Q)$. We write $\mathcal{C}^{\Diff}_\Q$ for the group of knots in $S^3$ considered up to rational smooth concordance.

Our main tool in this paper for obstructing the existence of smooth concordances will be the $\tau$ invariant~\cite{MR2026543}. The $\tau$ invariant is a group homomorphism $\tau\colon \mathcal{C}^{\Diff}_\Q\to\Z$, and via the natural forgetful homomorphism $\mathcal{C}^{\Diff}\to \mathcal{C}^{\Diff}_\Q$ it is also an invariant of smooth concordance in $S^3$. It has the property that $\tau(J)=\tau(J^r)$, where~$-^r$ denotes the operation of taking the knot with the reversed orientation. The \emph{smooth slice genus} of a knot~$J\subset S^3$ is the minimal integer $g$ such that there is a smoothly embedded oriented genus $g$ surface $\Sigma\subset D^4$ with boundary $J$. The value $|\tau(J)|$ is a lower bound for the smooth slice genus of~$J$~\cite[Corollary 1.3]{MR2026543}. It follows that~$|\tau(J)|$ is also a lower bound for the Seifert genus of $J$. It is shown in~\cite[Corollary 1.7]{MR2026543} that if $r,s$ are coprime integers with $r,s>1$, then there is the following Ozsv{\'a}th-Szab{\'o} \emph{torus knot formula} for the $\tau$ invariant:~$\tau(T_{r,s})=(r-1)(s-1)/2$.

We recall a theorem about the $\tau$ invariant of positive Whitehead doubles, due to Hedden \cite[Theorem 1.4]{Hedden:2007-1}, which we shall use later.

\begin{theorem}[Hedden]\label{thm:hedden}The $\tau$ invariant of the positive Whitehead double of a knot $J$ in $S^3$ is given by\[\tau(\Wh^+(J))=\left\{\begin{array}{ll}0&\text{for $\tau(J)\leq 0$,}\\1&\text{for $\tau(J)>0$.}\end{array}\right.\]
\end{theorem}

In \cite[Corollary 3]{MR2629765}, Van Cott generalises previous cabling results of Hedden~\cite{MR2511910} to derive a general formula for the $\tau$ invariant of the cable of a knot~$J\subset S^3$ in the case that $g_3(J)=\pm \tau(J)$.

\begin{proposition}[Van Cott]\label{prop:vancottcable}Let $J\subset S^3$ be a non trivial knot and let $r,s$ be coprime integers with~$r>1$.
\begin{enumerate}[leftmargin=*]
\item \label{item:vancottcable1} If $g_3(J)= \tau(J)$ then $\displaystyle{\tau(J_{r,s})=r\cdot \tau(J)+\frac{(r-1)(s-1)}{2}}$.
\item \label{item:vancottcable2} If $g_3(J)=- \tau(J)$ then $\displaystyle{\tau(J_{r,s})=r\cdot \tau(J)+\frac{(r-1)(s+1)}{2}}$.
\end{enumerate}
\end{proposition}

\begin{example}\label{exm:SmoothFamily}
Consider $\Wh^+(T_{2,3})$, the positive Whitehead double of the right-handed trefoil $T_{2,3}$. The Ozsv{\'a}th-Szab{\'o} torus knot formula gives~$\tau(T_{2,3}) = 1$. Note as well that there exists a genus 1 Seifert surface for~$\Wh^+(T_{2,3})$. As $|\tau|$ is a lower bound for the Seifert genus, these facts combine to show that $g_3(\Wh^+(T_{2,3}))=1$.

Now consider the family of knots~$J_i:=\Wh^+(T_{2,3})\#\cdots\#\Wh^+(T_{2,3})$, the $i$--fold connected sum. As $\tau(T_{2,3}) = 1$, by Theorem~\ref{thm:hedden}, $\tau(\Wh^+(T_{2,3}))=1$ and so~$\tau(J_i)=i$ and the $J_i$ are mutually distinct in smooth concordance.
The knots~$J_i$ are all topological null concordant as they have Alexander polynomial~$1$; see
e.g.~\cite[Theorem 11.7 B]{Freedman-Quinn:1990-1}.
\end{example}

\section{$\Diff$ versus\ $\Top$ for the null homotopic class in any $3$--manifold}

\noindent Assume for this section that $Y\neq S^3$. Our goal is to prove the following theorem.
\begin{theoremA}
For every $3$--manifold $Y \neq S^3$, there exists an infinite family of null homotopic knots $K_1, K_2, \dots$ in $Y$ that are all topologically concordant to one another but are mutually distinct in smooth almost concordance.
\end{theoremA}

\noindent The theorem will be divided into two cases:\begin{enumerate}
\item The case when $Y\neq \R P^3$, proved in Proposition~\ref{prop:heddenfamily}.
\item The case when $Y=\R P^3$, proved in Proposition~\ref{prop:exceptionalcase}.
\end{enumerate}

\subsection{Proof of Theorem~\ref{thm:A} when $Y\neq \R P^3$}

\begin{definition}
Suppose $K\subset Y$ and $J\subset S^3$ are knots and that $\psi$ is a choice of framing for $K$. Then we define a knot $K_J\subset Y$ as the image of the knot~$P_J\subset S^1\times D^2$ from Figure~\ref{fig:pattern1} under the embedding $S^1\times D^2\subset Y$ determined by the framing~$\psi$. (The framing is intentionally suppressed from the notation, for brevity.)\begin{figure}[h]

\begin{tikzpicture}

\begin{scope}[shift={(0,0)}, yscale=0.7,xscale=0.7]\begin{knot}[
clip width=5,
]

\strand [black, thick] (0,0.3)
to [out=up, in=left, looseness=0.8] (1,1.4)
to [out=right, in=up, looseness=0.8] (2.3,1.1)
to [out=down, in=right, looseness=0.8] (1,0.8)
to [out=left, in=left, looseness=0.8] (1,-0.8)
to [out=right, in=up, looseness=0.8] (2.3,-1.1)
to [out=down, in=right, looseness=0.8] (1,-1.4)
to [out=left, in=down] (0,-0.3);

\strand [black, thick] (5,0.3)
to [out=up, in=right, looseness=0.8] (4,1.2)
to [out=left, in=up, looseness=0.6] (1.7,0.9)
to [out=down, in=left, looseness=0.6] (4,0.6)
to [out=right, in=right, looseness=0.6] (4,-0.6)
to [out=left, in=up, looseness=0.6] (1.7,-0.9)
to [out=down, in=left, looseness=0.8] (4,-1.2)
to [out=right, in=down] (5,-0.3);

\strand [black, thick] (5,0.3)
to [out=right, in=right, looseness=3] (2.5,-3.4)
to [out=left, in=left, looseness=3] (0,0.3);

\strand [black, thick] (5,-0.3)
to [out=right, in=right, looseness=3] (2.5,-2.8)
to [out=left, in=left, looseness=3] (0,-0.3);

\strand [black, ultra thin] (9,-1)
to [out=down, in=right, looseness=0.8] (2.5,-4)
to [out=left, in=down, looseness=0.8] (-4,-1)
to [out=up, in=left, looseness=0.8] (2.5,1.8)
to [out=right, in=up, looseness=0.8] (9,-1);

\strand [black, ultra thin] (4,-2)
to [out=250, in=0, looseness=0.3] (3,-2.3)
to [out=180, in=0, looseness=0.5] (2.5,-2.32)
to [out=180, in=0, looseness=0.5] (2,-2.3)
to [out=180, in=290, looseness=0.3] (1,-2);

\strand [black, ultra thin] (3,-2.3)
to [out=110, in=0] (2.5,-2)
to [out=180, in=70] (2,-2.3);

\flipcrossings{2,3}
\end{knot}

\begin{scope}[xscale=0.5, yscale=0.5, shift={(7,2.45)}]
\draw[thick] (-0.5, 0.5) -- (0,0);
\draw[thick] (-0.5, -0.5) -- (0,0);
\end{scope}

\draw [thick, draw=black, fill=white]
       (2.7,-1.5) -- (2.7,-0.4) -- (3.8,-0.4) -- (3.8,-1.5) -- cycle;
\node at (3.25,-0.95) {$J$};

\end{scope}

\end{tikzpicture}

\caption{A knot $P_J$ in the solid torus, where the $J$--box denotes tying the strands into the knot $J$.}
\label{fig:pattern1}
\end{figure}
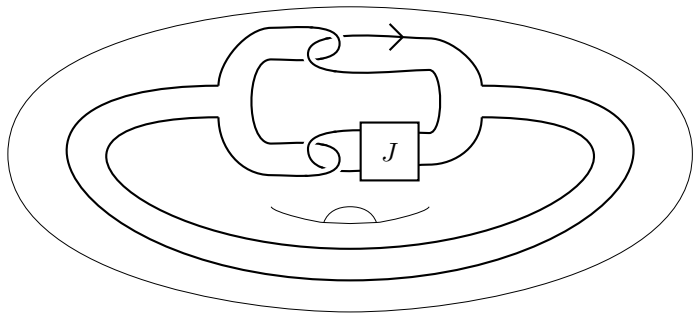

\end{definition}

We may undo the clasps of the knot $P_J$ in Figure~\ref{fig:pattern1} by means of a free homotopy, showing that $P_J$ is freely null homotopic in $S^1\times D^2$. It follows that each knot~$K_J$ is freely null homotopic in $Y$.

Let $K\subset Y$ be a knot which is homotopically essential and choose $J\subset S^3$ and a framing $\psi$ of $K$. In the universal cover $\pi\colon \wt{Y}\to Y$ write $t\colon \wt{Y}\to\wt{Y}$ for the deck transformation corresponding to $[K]\in\pi_1(Y)$. As $K_J$ is null homotopic, the discussion in Section~\ref{subsec:coveringlinks} implies that for each component $\wt{K}$ of $\pi^{-1}(K_J)$, the degree of the covering $\wt{K}\to K_J$ is 1. Moreover, as $[K]$ is nontrivial in $\pi_1(Y)$, the deck transformation $t$ is nontrivial. This justifies the following definition.

\begin{definition}
Any choice of 2--component link $\wt{K} \cup t(\wt{K})\subset \pi^{-1}(K_J)\subset \wt{Y}$ is called a \emph{covering double} of $K_J$.
\end{definition}

\noindent Here is a key observation.

\begin{proposition}\label{prop:BD}Let $K\subset Y$ and let $\pi\colon \wt{Y}\to Y$ be the universal cover. Write the deck transformation corresponding to $[K]\in\pi_1(Y)$ as $t\colon \wt{Y}\to \wt{Y}$ and assume that the order of $t$ is greater than two. Let $\psi$ be any framing of $K$, let $J\subset S^3$ be a knot, and let $\wt{K}\cup t(\wt{K})$ be a covering double of $K_J$. Remove a $3$--disc from~$Y$ away from~$K_J$ and embed $\phi\colon \pi^{-1}(\cl(Y\sm D^3))\hookrightarrow S^3$ using Lemma~\ref{lem:boden-nagel}. Then the~$2$--component link $\phi(\wt{K}\cup t(\wt{K}))\subset S^3$ is the Bing double $\BD(J)$.
\end{proposition}

\begin{proof}
See Figure~\ref{fig:daisychain} for a picture of the embedded preimage $\phi(\pi^{-1}(K_J))\subset S^3$. The ?--boxes indicate the embedding $\phi$ may knot the `bands' between the components in an unknown manner in $S^3$. But as this unknown knotting occurs band-wise it can be removed by an isotopy when we restrict to a (2--component) covering double.
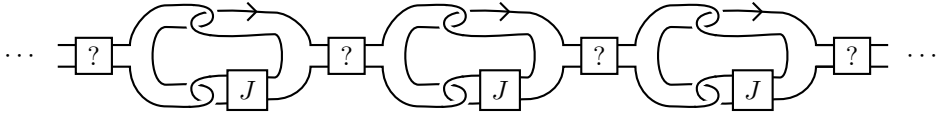
\begin{figure}[h]

\begin{tikzpicture}[scale=0.48]

\begin{scope}[shift={(0,0)}]

\begin{scope}[xscale=0.5, yscale=0.5, shift={(7,2.45)}]
\draw[thick] (-0.5, 0.5) -- (0,0);
\draw[thick] (-0.5, -0.5) -- (0,0);
\end{scope}

\begin{knot}[
clip radius=7pt,
clip width=5,
]

\strand [black, thick] (-2,0.3)
to [out=right, in=left] (0,0.3)
to [out=up, in=left, looseness=0.8] (1,1.4)
to [out=right, in=up, looseness=0.8] (2.3,1.1)
to [out=down, in=right, looseness=0.8] (1,0.8)
to [out=left, in=left, looseness=0.8] (1,-0.8)
to [out=right, in=up, looseness=0.8] (2.3,-1.1)
to [out=down, in=right, looseness=0.8] (1,-1.4)
to [out=left, in=down] (0,-0.3)
to [out=left, in=right] (-2,-0.3);

\strand [black, thick] (6,0.3)
to [out=left, in=right] (5,0.3)
to [out=up, in=right, looseness=0.8] (4,1.2)
to [out=left, in=up, looseness=0.6] (1.7,0.9)
to [out=down, in=left, looseness=0.6] (4,0.6)
to [out=right, in=right, looseness=0.6] (4,-0.6)
to [out=left, in=up, looseness=0.6] (1.7,-0.9)
to [out=down, in=left, looseness=0.8] (4,-1.2)
to [out=right, in=down] (5,-0.3)
to [out=right, in=left] (6,-0.3);

\flipcrossings{2,3}
\end{knot}

\draw [thick, draw=black, fill=white]
       (2.7,-1.5) -- (2.7,-0.4) -- (3.8,-0.4) -- (3.8,-1.5) -- cycle;
\node at (3.25,-0.95) {$J$};

\end{scope}

\begin{scope}[shift={(7,0)}]

\begin{scope}[xscale=0.5, yscale=0.5, shift={(7,2.45)}]
\draw[thick] (-0.5, 0.5) -- (0,0);
\draw[thick] (-0.5, -0.5) -- (0,0);
\end{scope}

\begin{knot}[
clip radius=7pt,
clip width=5,
]

\strand [black, thick] (-1,0.3)
to [out=right, in=left] (0,0.3)
to [out=up, in=left, looseness=0.8] (1,1.4)
to [out=right, in=up, looseness=0.8] (2.3,1.1)
to [out=down, in=right, looseness=0.8] (1,0.8)
to [out=left, in=left, looseness=0.8] (1,-0.8)
to [out=right, in=up, looseness=0.8] (2.3,-1.1)
to [out=down, in=right, looseness=0.8] (1,-1.4)
to [out=left, in=down] (0,-0.3)
to [out=left, in=right] (-1,-0.3);

\strand [black, thick] (6,0.3)
to [out=left, in=right] (5,0.3)
to [out=up, in=right, looseness=0.8] (4,1.2)
to [out=left, in=up, looseness=0.6] (1.7,0.9)
to [out=down, in=left, looseness=0.6] (4,0.6)
to [out=right, in=right, looseness=0.6] (4,-0.6)
to [out=left, in=up, looseness=0.6] (1.7,-0.9)
to [out=down, in=left, looseness=0.8] (4,-1.2)
to [out=right, in=down] (5,-0.3)
to [out=right, in=left] (6,-0.3);

\flipcrossings{2,3}
\end{knot}

\draw [thick, draw=black, fill=white]
       (2.7,-1.5) -- (2.7,-0.4) -- (3.8,-0.4) -- (3.8,-1.5) -- cycle;
\node at (3.25,-0.95) {$J$};

\end{scope}

\begin{scope}[shift={(14,0)}]

\begin{scope}[xscale=0.5, yscale=0.5, shift={(7,2.45)}]
\draw[thick] (-0.5, 0.5) -- (0,0);
\draw[thick] (-0.5, -0.5) -- (0,0);
\end{scope}

\begin{knot}[
clip radius=7pt,
clip width=5,
]

\strand [black, thick] (-1,0.3)
to [out=right, in=left] (0,0.3)
to [out=up, in=left, looseness=0.8] (1,1.4)
to [out=right, in=up, looseness=0.8] (2.3,1.1)
to [out=down, in=right, looseness=0.8] (1,0.8)
to [out=left, in=left, looseness=0.8] (1,-0.8)
to [out=right, in=up, looseness=0.8] (2.3,-1.1)
to [out=down, in=right, looseness=0.8] (1,-1.4)
to [out=left, in=down] (0,-0.3)
to [out=left, in=right] (-1,-0.3);

\strand [black, thick] (6,0.3)
to [out=left, in=right] (5,0.3)
to [out=up, in=right, looseness=0.8] (4,1.2)
to [out=left, in=up, looseness=0.6] (1.7,0.9)
to [out=down, in=left, looseness=0.6] (4,0.6)
to [out=right, in=right, looseness=0.6] (4,-0.6)
to [out=left, in=up, looseness=0.6] (1.7,-0.9)
to [out=down, in=left, looseness=0.8] (4,-1.2)
to [out=right, in=down] (5,-0.3)
to [out=right, in=left] (6,-0.3);

\flipcrossings{2,3}
\end{knot}

\draw [thick, draw=black, fill=white]
       (2.7,-1.5) -- (2.7,-0.4) -- (3.8,-0.4) -- (3.8,-1.5) -- cycle;
\node at (3.25,-0.95) {$J$};

\end{scope}

%
%
%
%
%

\draw[thick, draw=black] (19,0.3)
to [out=right, in=left] (21,0.3);

\draw[thick, draw=black] (19,-0.3)
to [out=right, in=left] (21,-0.3);

\begin{scope}[shift={(-1,0)}]
\draw [thick, draw=black, fill=white]
       (-0.5,-0.5) -- (0.5,-0.5) -- (0.5,0.5) -- (-0.5,0.5) -- cycle;
\node at (0,0) {?};
\end{scope}

\begin{scope}[shift={(6,0)}]
\draw [thick, draw=black, fill=white]
       (-0.5,-0.5) -- (0.5,-0.5) -- (0.5,0.5) -- (-0.5,0.5) -- cycle;
\node at (0,0) {?};
\end{scope}

\begin{scope}[shift={(13,0)}]
\draw [thick, draw=black, fill=white]
       (-0.5,-0.5) -- (0.5,-0.5) -- (0.5,0.5) -- (-0.5,0.5) -- cycle;
\node at (0,0) {?};
\end{scope}

\begin{scope}[shift={(20,0)}]
\draw [thick, draw=black, fill=white]
       (-0.5,-0.5) -- (0.5,-0.5) -- (0.5,0.5) -- (-0.5,0.5) -- cycle;
\node at (0,0) {?};
\end{scope}

\node at (-3,0) {$\cdots$};

\node at (22,0) {$\cdots$};

\end{tikzpicture}
\caption{The preimage $\pi^{-1}(K_J)$ embedded in $S^3$ via $\phi$. The $\text{?}$--boxes indicate the possibility that $\phi$ may knot the `bands' in an unknown manner.}
\label{fig:daisychain}
\end{figure}

\end{proof}

\noindent The following corollary is immediate.

\begin{corollary}Assume the hypotheses of Proposition~\ref{prop:BD} and let $J_1, J_2$ be knots in $S^3$. If the covering doubles $\wt{K_i}\cup t(\wt{K_i})$ for $K_{J_i}$, $i=1,2$ are concordant to one another  in $\wt{Y}$, then $\BD(J_1)$ is concordant to $\BD(J_2)$.
\end{corollary}

Note that the conclusion of Proposition~\ref{prop:BD} is independent of the choice of $\psi$ that was made.

\begin{lemma}\label{lem:mostwork}Let $K\subset Y$ such that $[K]\in\pi_1(Y)$ is of order greater than 2, and fix a framing $\psi$ for $K$. For any two knots $J,J'\subset S^3$, if $K_J,K_{J'}\subset Y$ are smoothly almost concordant then the Bing doubles~$\BD(J)$ and $\BD(J')$ are smoothly concordant.
\end{lemma}

\begin{proof}The preimage of any $K_J$ in the universal cover $\wt{Y}$ under the covering map, has unknotted components. Suppose $K_J,K_{J'}$ are smoothly almost concordant. As their preimages in $\wt{Y}$ have unknotted components, Proposition~\ref{prop:almostimpliesconc} implies that these preimages are concordant in $\wt{Y}$. In particular, any choice of covering doubles~$L_1\cup L_2$ and $L_1'\cup L_2'$, for $K_J$ and $K_{J'}$ respectively, are concordant. So, by Proposition~\ref{prop:BD},~$\BD(J)$ is concordant to $\BD(J')$.
\end{proof}

A crucial hypothesis for Proposition~\ref{prop:BD} is that the order of the deck transformation $t$ is greater than two. The $3$--manifolds with no elements of order greater than 2 in their fundamental group form a very short list.

\begin{proposition}\label{prop:RP3}If $Y$ is a $3$--manifold, then either $\pi_1(Y)$ contains an element of order greater than 2 or $Y\in\{S^3, \,\R P^3\}$.
\end{proposition}

\begin{proof}
If all the nontrivial elements of $\pi_1(Y)$ are of order two then the group is abelian because for any two elements $a,b$, we have $a b a^{-1} b^{-1}=a b a b=(ab)^2=1$. Hence since $\pi_1(Y)$ is finitely generated, the group must be finite.  The universal cover of $Y$ is therefore a simply connected closed 3--manifold, thus is $S^3$ by the geometrisation theorem. The possible fundamental groups of $Y$ are then equal to the finite abelian groups that act freely and isometrically on $S^3$.  The possible finite groups were determined in~\cite[\textsection 2]{MR1512281} and are listed in \cite[Section~1.7]{MR3444187}. Of these groups, the only nontrivial abelian groups are the cyclic groups~$C_p$ for~$p \geq 2$.
 Of course only the cyclic group $C_2$ fails to contain an element of order greater than two.  In particular we point out that $\prod_m C_2$ does not occur as the fundamental group of a closed 3--manifold for $m \geq 2$, as shown in~\cite[Section~1.11,~Table~1.2]{MR3444187}.
The unique 3--manifold with fundamental group isomorphic to $C_2$ is $\R P^3$.
\end{proof}

Fix a $3$--manifold $Y$ which is equal neither to $S^3$, nor to $\R P^3$. Choose a class~$g\in\pi_1(Y)$ of order at least three. Fix a representative knot $K\subset Y$ of $g$ and choose a framing $\psi$ for $K$. Write $D=\Wh^+(T_{2,3})$ and define a family of knots
\[\mathcal{F}:=\{K_J\subset Y\, |\, J\in\{nD\}_{n=1}^\infty\}.\] Recall from Example~\ref{exm:SmoothFamily} that the knots $nD$ are all topologically null concordant. Hence the knots in $\mathcal{F}$ are all topologically concordant to $K_U$, where $U\subset S^3$ is the unknot, and thus are mutually topologically concordant. In particular, the knots in $\mathcal{F}$ are mutually topologically almost concordant.

The next proposition completes the proof of Theorem~\ref{thm:A} when $Y\neq \R P^3$.

\begin{proposition}\label{prop:heddenfamily}
The knots in the infinite family $\mathcal{F}$ are mutually distinct in smooth almost concordance.
\end{proposition}

\begin{proof}
We recall an observation from Cha-Livingston-Ruberman \cite{Cha-Livingston-Ruberman:2006-1}. Given a knot $J\subset S^3$, consider the double cover of $S^3$ branched over a component of the Bing double $\BD(J)$. Write $J^r$ for $J$ with the reversed orientation. The preimage of the other component of $\BD(J)$ under the covering map is a 2--component link with each component equal to $J\#J^r$. Now if there is a smooth concordance~$A_1\cup A_2\subset S^3\times [0,1]$ from $\BD(J)$ to $\BD(J')$, we may consider the double cover~$\pi\colon M\to S^3\times [0,1]$, branched over a single component of that concordance,~$A_2$, say. A Meier-Vietoris argument shows $H_*(M;\Q)=H_*(S^3;\Q)$, and hence a single component of the preimage $\pi^{-1}(A_1)$ is a rational smooth concordance from~$J\#J^r$ to~$J'\#(J')^r$. As the~$\tau$ invariant is a rational smooth concordance invariant, a homomorphism, and does not change under orientation reversal, this implies that~$\tau(J)=\tau(J')$. Combining this observation with Lemma~\ref{lem:mostwork}, it now suffices to recall from Example~\ref{exm:SmoothFamily} that the knots $nD \in S^3$, for $n>0$, have the property that~$\tau(nD)\neq \tau(mD)$ when $n\neq m$.
\end{proof}

We remark that the only features required of the family of knots from Example~\ref{exm:SmoothFamily} in order that they could be used in the proof above were that they are all topologically concordant to the unknot, and that $\tau(nD)\neq \tau(mD)$ when $n\neq m$. As such, any family with these properties could have been used to define $\mathcal{F}$ and prove the proposition.

\subsection{Proof of Theorem~\ref{thm:A} when $Y=\R P^3$}
\label{sec:RP3}
Let $J\subset S^3$ be a knot and now write~$K_J\subset \R P^3$ and $L_J\subset S^3$, to denote respectively the knot and the covering link depicted in Figure~\ref{fig:rp3knot}. The proof that $L_J$ is the covering link of $K_J$ follows from Lemma~\ref{lem:coveringlink-lens-spaces} as $\R P^3= L(2,1)$.

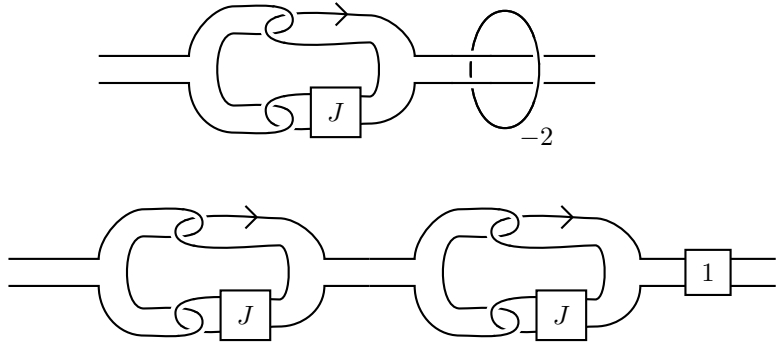
\begin{figure}[h]

\begin{tikzpicture}[scale=0.6]

\begin{scope}[shift={(2,4.5)}]

\begin{scope}[xscale=0.5, yscale=0.5, shift={(7,2.45)}]
\draw[thick] (-0.5, 0.5) -- (0,0);
\draw[thick] (-0.5, -0.5) -- (0,0);
\end{scope}

\begin{knot}[
clip radius=7pt,
clip width=5,
]

\strand [black, thick] (-2,0.3)
to [out=right, in=left] (0,0.3)
to [out=up, in=left, looseness=0.8] (1,1.4)
to [out=right, in=up, looseness=0.8] (2.3,1.1)
to [out=down, in=right, looseness=0.8] (1,0.8)
to [out=left, in=left, looseness=0.8] (1,-0.8)
to [out=right, in=up, looseness=0.8] (2.3,-1.1)
to [out=down, in=right, looseness=0.8] (1,-1.4)
to [out=left, in=down] (0,-0.3)
to [out=left, in=right] (-2,-0.3);

\strand [black, thick] (9,0.3)
to [out=left, in=right] (5,0.3)
to [out=up, in=right, looseness=0.8] (4,1.2)
to [out=left, in=up, looseness=0.6] (1.7,0.9)
to [out=down, in=left, looseness=0.6] (4,0.6)
to [out=right, in=right, looseness=0.6] (4,-0.6)
to [out=left, in=up, looseness=0.6] (1.7,-0.9)
to [out=down, in=left, looseness=0.8] (4,-1.2)
to [out=right, in=down] (5,-0.3)
to [out=right, in=left] (9,-0.3);

\strand[thick, draw=black] (7,1.3)
to [out=right, in=right] (7,-1.3)
to [out=left, in=left] (7,1.3);

\flipcrossings{2,3,5,7}
\end{knot}

\draw [thick, draw=black, fill=white]
       (2.7,-1.5) -- (2.7,-0.4) -- (3.8,-0.4) -- (3.8,-1.5) -- cycle;
\node at (3.25,-0.95) {$J$};

\node at (7.7,-1.5) {$-2$};

\end{scope}

\begin{scope}[shift={(0,0)}]

\begin{scope}[xscale=0.5, yscale=0.5, shift={(7,2.45)}]
\draw[thick] (-0.5, 0.5) -- (0,0);
\draw[thick] (-0.5, -0.5) -- (0,0);
\end{scope}

\begin{knot}[
clip radius=7pt,
clip width=5,
]

\strand [black, thick] (-2,0.3)
to [out=right, in=left] (0,0.3)
to [out=up, in=left, looseness=0.8] (1,1.4)
to [out=right, in=up, looseness=0.8] (2.3,1.1)
to [out=down, in=right, looseness=0.8] (1,0.8)
to [out=left, in=left, looseness=0.8] (1,-0.8)
to [out=right, in=up, looseness=0.8] (2.3,-1.1)
to [out=down, in=right, looseness=0.8] (1,-1.4)
to [out=left, in=down] (0,-0.3)
to [out=left, in=right] (-2,-0.3);

\strand [black, thick] (6,0.3)
to [out=left, in=right] (5,0.3)
to [out=up, in=right, looseness=0.8] (4,1.2)
to [out=left, in=up, looseness=0.6] (1.7,0.9)
to [out=down, in=left, looseness=0.6] (4,0.6)
to [out=right, in=right, looseness=0.6] (4,-0.6)
to [out=left, in=up, looseness=0.6] (1.7,-0.9)
to [out=down, in=left, looseness=0.8] (4,-1.2)
to [out=right, in=down] (5,-0.3)
to [out=right, in=left] (6,-0.3);

\flipcrossings{2,3}
\end{knot}

\draw [thick, draw=black, fill=white]
       (2.7,-1.5) -- (2.7,-0.4) -- (3.8,-0.4) -- (3.8,-1.5) -- cycle;
\node at (3.25,-0.95) {$J$};

\end{scope}

\begin{scope}[shift={(7,0)}]

\begin{scope}[xscale=0.5, yscale=0.5, shift={(7,2.45)}]
\draw[thick] (-0.5, 0.5) -- (0,0);
\draw[thick] (-0.5, -0.5) -- (0,0);
\end{scope}

\begin{knot}[
clip radius=7pt,
clip width=5,
]

\strand [black, thick] (-1,0.3)
to [out=right, in=left] (0,0.3)
to [out=up, in=left, looseness=0.8] (1,1.4)
to [out=right, in=up, looseness=0.8] (2.3,1.1)
to [out=down, in=right, looseness=0.8] (1,0.8)
to [out=left, in=left, looseness=0.8] (1,-0.8)
to [out=right, in=up, looseness=0.8] (2.3,-1.1)
to [out=down, in=right, looseness=0.8] (1,-1.4)
to [out=left, in=down] (0,-0.3)
to [out=left, in=right] (-1,-0.3);

\strand [black, thick] (8,0.3)
to [out=left, in=right] (5,0.3)
to [out=up, in=right, looseness=0.8] (4,1.2)
to [out=left, in=up, looseness=0.6] (1.7,0.9)
to [out=down, in=left, looseness=0.6] (4,0.6)
to [out=right, in=right, looseness=0.6] (4,-0.6)
to [out=left, in=up, looseness=0.6] (1.7,-0.9)
to [out=down, in=left, looseness=0.8] (4,-1.2)
to [out=right, in=down] (5,-0.3)
to [out=right, in=left] (8,-0.3);

\flipcrossings{2,3}
\end{knot}

\draw [thick, draw=black, fill=white]
       (2.7,-1.5) -- (2.7,-0.4) -- (3.8,-0.4) -- (3.8,-1.5) -- cycle;
\node at (3.25,-0.95) {$J$};

\end{scope}

\begin{scope}[shift={(13.5,0)}]
\draw [thick, draw=black, fill=white]
       (-0.5,-0.5) -- (0.5,-0.5) -- (0.5,0.5) -- (-0.5,0.5) -- cycle;
\node at (0,0) {$1$};
\end{scope}

\end{tikzpicture}
\caption{The knot $K_J$ in $\R P^3$, and its covering link $L_J$ in $S^3$.}
\label{fig:rp3knot}
\end{figure}

Label the components of $L_J$ as $L_1$ and $L_2$. As $L_1$ is unknotted, the double branched cover of $S^3$ over $L_1$ is again $S^3$. Write $\wt{L}_J\subset S^3$ for the preimage of $L_2$ under this double branched cover. In Figure~\ref{fig:manipulate} we move the covering link $L_J$ in to a new form by an isotopy. From this new form it is straightforward to see that $\wt{L}_J$ is the second link drawn in Figure~\ref{fig:manipulate}.

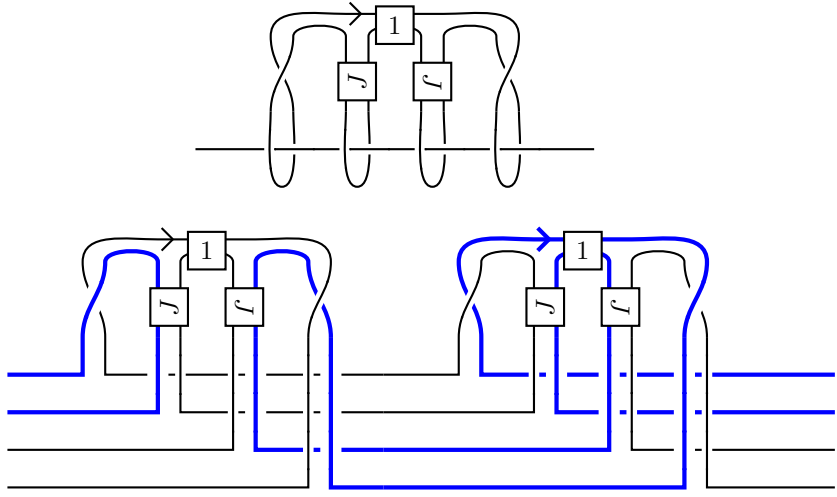
\begin{figure}[h]

\begin{tikzpicture}[scale=1]

\begin{scope}[shift={(0,0)}]

\begin{scope}[xscale=0.3, yscale=0.3, shift={(4,7.67)}]
\draw[thick] (-0.5, 0.5) -- (0,0);
\draw[thick] (-0.5, -0.5) -- (0,0);
\end{scope}

\begin{knot}[
clip radius=7pt,
clip width=5,
]

\strand [black, thick] (0.3,1)
to [out=up, in=down] (0,2)
to [out=up, in=left] (1,2.3)
to [out=right, in=left] (2.3,2.3)
to [out=right, in=up] (3.3,2)
to [out=down, in=up] (3,1);

\strand [black, thick] (0,1)
to [out=up, in=down] (0.3,2)
to [out=up, in=up, looseness=0.7] (1,2)
to [out=down, in=up] (1,1);

\strand [black, thick] (1.3,1)
to [out=up, in=down] (1.3,2)
to [out=up, in=up, looseness=0.7] (2,2)
to [out=down, in=up] (2,1);

\strand [black, thick] (2.3,1)
to [out=up, in=down] (2.3,2)
to [out=up, in=up, looseness=0.7] (3,2)
to [out=down, in=up] (3.3,1);

\strand [black, thick] (0,1)
to [out=down, in=left, looseness=0.6] (0.15,0)
to [out=right, in=down, looseness=0.6] (0.3,1);

\strand [black, thick] (1,1)
to [out=down, in=left, looseness=0.6] (1.15,0)
to [out=right, in=down, looseness=0.6] (1.3,1);

\strand [black, thick] (2,1)
to [out=down, in=left, looseness=0.6] (2.15,0)
to [out=right, in=down, looseness=0.6] (2.3,1);

\strand [black, thick] (3,1)
to [out=down, in=left, looseness=0.6] (3.15,0)
to [out=right, in=down, looseness=0.6] (3.3,1);

\strand[black, thick] (-1,0.5)
to [out=right, in=left] (4.3,0.5);

\flipcrossings{1, 4, 6, 8, 10}
\end{knot}

\draw [thick, draw=black, fill=white]
       (1.4,1.9) -- (1.9,1.9) -- (1.9,2.4) -- (1.4,2.4) -- cycle;
\node at (1.65, 2.15) {$1$};

\draw [thick, draw=black, fill=white]
       (0.9,1.15) -- (1.4,1.15) -- (1.4,1.65) -- (0.9,1.65) -- cycle;
\node[rotate=90] at (1.15,1.4) {$J$};

\draw [thick, draw=black, fill=white]
       (1.9,1.15) -- (2.4,1.15) -- (2.4,1.65) -- (1.9,1.65) -- cycle;
\node[rotate=-90] at (2.15,1.4) {$J$};

\end{scope}

\begin{scope}[shift={(-2.5,-3)}]

\begin{scope}[xscale=0.3, yscale=0.3, shift={(4,7.67)}]
\draw[thick] (-0.5, 0.5) -- (0,0);
\draw[thick] (-0.5, -0.5) -- (0,0);
\end{scope}

\begin{knot}[
clip radius=7pt,
clip width=5,
]

\strand [black, thick] (4,0.5)
to [out=left, in=right] (0.3,0.5)
to [out=up, in=down] (0.3,1)
to [out=up, in=down] (0,2)
to [out=up, in=left] (1,2.3)
to [out=right, in=left] (2.3,2.3)
to [out=right, in=up] (3.3,2)
to [out=down, in=up] (3,1);

\strand [black, thick] (3,1)
to [out=down, in=up] (3,-1)
to [out=left, in=right] (-1,-1);

\strand [blue, ultra thick] (-1,0.5)
to [out=right, in=left] (0,0.5)
to [out=up, in=down] (0,1)
to [out=up, in=down] (0.3,2)
to [out=up, in=up, looseness=0.7] (1,2)
to [out=down, in=up] (1,1);

\strand[blue, ultra thick] (1,1)
to [out=down, in=up] (1,0)
to [out=left, in=right] (-1,0);

\strand [black, thick] (1.3,1)
to [out=up, in=down] (1.3,2)
to [out=up, in=up, looseness=0.7] (2,2)
to [out=down, in=up] (2,1);

\strand[black, thick] (2,1)
to [out=down, in=up] (2,-0.5)
to [out=left, in=right] (-1,-0.5);

\strand[blue, ultra thick] (2.3,1)
to [out=down, in =up] (2.3,-0.5)
to [out=right, in=left] (4, -0.5);

\strand [blue, ultra thick] (2.3,1)
to [out=up, in=down] (2.3,2)
to [out=up, in=up, looseness=0.7] (3,2)
to [out=down, in=up] (3.3,1);

\strand[blue, ultra thick] (3.3,1)
to [out=down, in =up] (3.3,-1)
to [out=right, in=left] (4, -1);

\strand [black, thick] (4,0)
to [out=left, in=right] (1.3,0)
to [out=up, in=down] (1.3, 1);

\flipcrossings{1,2,3,4,5,7,8,12}
\end{knot}

\draw [thick, draw=black, fill=white]
       (1.4,1.9) -- (1.9,1.9) -- (1.9,2.4) -- (1.4,2.4) -- cycle;
\node at (1.65, 2.15) {$1$};

\draw [thick, draw=black, fill=white]
       (0.9,1.15) -- (1.4,1.15) -- (1.4,1.65) -- (0.9,1.65) -- cycle;
\node[rotate=90] at (1.15,1.4) {$J$};

\draw [thick, draw=black, fill=white]
       (1.9,1.15) -- (2.4,1.15) -- (2.4,1.65) -- (1.9,1.65) -- cycle;
\node[rotate=-90] at (2.15,1.4) {$J$};

\end{scope}

\begin{scope}[shift={(2.5,-3)}]

\begin{scope}[xscale=0.3, yscale=0.3, shift={(4,7.67)}]
\draw[blue, ultra thick] (-0.5, 0.5) -- (0,0);
\draw[blue, ultra thick] (-0.5, -0.5) -- (0,0);
\end{scope}

\begin{knot}[
clip radius=7pt,
clip width=5,
]

\strand [blue, ultra thick] (5,0.5)
to [out=left, in=right] (0.3,0.5)
to [out=up, in=down] (0.3,1)
to [out=up, in=down] (0,2)
to [out=up, in=left] (1,2.3)
to [out=right, in=left] (2.3,2.3)
to [out=right, in=up] (3.3,2)
to [out=down, in=up] (3,1);

\strand [blue, ultra thick] (3,1)
to [out=down, in=up] (3,-1)
to [out=left, in=right] (-1,-1);

\strand [black, thick] (-1,0.5)
to [out=right, in=left] (0,0.5)
to [out=up, in=down] (0,1)
to [out=up, in=down] (0.3,2)
to [out=up, in=up, looseness=0.7] (1,2)
to [out=down, in=up] (1,1);

\strand[black, thick] (1,1)
to [out=down, in=up] (1,0)
to [out=left, in=right] (-1,0);

\strand [blue, ultra thick] (1.3,1)
to [out=up, in=down] (1.3,2)
to [out=up, in=up, looseness=0.7] (2,2)
to [out=down, in=up] (2,1);

\strand[blue, ultra thick] (2,1)
to [out=down, in=up] (2,-0.5)
to [out=left, in=right] (-1,-0.5);

\strand[black, thick] (2.3,1)
to [out=down, in =up] (2.3,-0.5)
to [out=right, in=left] (5, -0.5);

\strand [black, thick] (2.3,1)
to [out=up, in=down] (2.3,2)
to [out=up, in=up, looseness=0.7] (3,2)
to [out=down, in=up] (3.3,1);

\strand[black, thick] (3.3,1)
to [out=down, in =up] (3.3,-1)
to [out=right, in=left] (5, -1);

\strand [blue, ultra thick] (5,0)
to [out=left, in=right] (1.3,0)
to [out=up, in=down] (1.3, 1);

\flipcrossings{1,2,3,4,5,7,8,12}
\end{knot}

\draw [thick, draw=black, fill=white]
       (1.4,1.9) -- (1.9,1.9) -- (1.9,2.4) -- (1.4,2.4) -- cycle;
\node at (1.65, 2.15) {$1$};

\draw [thick, draw=black, fill=white]
       (0.9,1.15) -- (1.4,1.15) -- (1.4,1.65) -- (0.9,1.65) -- cycle;
\node[rotate=90] at (1.15,1.4) {$J$};

\draw [thick, draw=black, fill=white]
       (1.9,1.15) -- (2.4,1.15) -- (2.4,1.65) -- (1.9,1.65) -- cycle;
\node[rotate=-90] at (2.15,1.4) {$J$};

\end{scope}

\end{tikzpicture}
\caption{The link $\wt{L}_J$ has 2 components, each a copy of $2J\#2(J^r)$.}
\label{fig:manipulate}
\end{figure}

Now select one of the 2 components of $\wt{L}_J$, for example we take the component highlighted in Figure~\ref{fig:manipulate}. By following this strand around, we see that this component is $2J\#2J^r$, where $-^r$ denotes the operation of taking the reversed orientation.

Write $D=\Wh^+(T_{2,3})$ and define a family of knots
\[\mathcal{F}':=\{K_J\subset Y\, |\, J\in\{nD\}_{n=1}^\infty\}.\]The knots in $\mathcal{F}'$ are mutually topologically concordant. In particular the knots in~$\mathcal{F}'$ are mutually topologically almost concordant.

\begin{proposition}\label{prop:exceptionalcase}
The knots in the infinite family $\mathcal{F}'$ are mutually distinct in smooth almost concordance.
\end{proposition}

\begin{proof}
Suppose $K_J,K_{J'}\in\mathcal{F'}$ are smoothly almost concordant. Write $\pi\colon S^3\to \R P^3$ for the universal cover. Then as $L_J$ and $L_J'$ have unknotted components we apply Proposition~\ref{prop:almostimpliesconc} to conclude that $L_J$ and $L_J'$ are smoothly concordant. Lifting to the double branched covers over the respective first components, there is a smooth concordance from $\wt{L}_J$ to $\wt{L}_J'$. Selecting one component of this concordance we obtain a rational concordance from $2J\#2J^r$ to $2J'\#2(J')^r$. As the $\tau$ invariant is a rational concordance invariant, a homomorphism, and does not change under orientation reversal, it is thus sufficient to recall that the knots $\mathcal{F'}$ were constructed using a family of knots $nD$ with distinct~$\tau$ invariants.\end{proof}

\section{$\Diff$ versus\ $\Top$ for any class in a lens space}

We will now show that, at least for lens spaces, the distinction between smooth and topological almost concordance is not restricted to the nullhomotopic class. Indeed, in every free homotopy class in a lens space there is a distinction.

\begin{theoremB}
  Let $Y = L(p,q)$, for $\gcd(p,q)=1$, $p>1$, and let $x \in [S^1, Y]$ be any free homotopy class.   Then there exist topologically concordant knots $K$ and $K'$ in~$Y$ representing the class $x$, that are distinct in smooth almost concordance.
\end{theoremB}

Note that the \emph{order} of a free homotopy class $x\in [S^1,Y]$ is always well-defined despite the fact that $[S^1,Y]$ is not a group. We divide Theorem~\ref{thm:B} into four cases:\begin{enumerate}
\item \label{item:case1} The case when $x$ is order 0 is covered by Theorem~\ref{thm:A}.
\item \label{item:case2}  The case when $x$ is neither order 0 nor 2, proved in Proposition~\ref{prop:lens_space_generic}.
\item \label{item:case3}  The case when $x$ is order 2 and $Y\neq \R P^3$, proved in Proposition~\ref{prop:L2nqnotRP3}.
\item \label{item:case4}  The case when $x$ is order 2 and $Y= \R P^3$, proved in Proposition~\ref{prop:RP3}.
\end{enumerate}

A particularly neat covering link argument is possible in case (\ref{item:case2}), the generic case. The argument is possible because of the way arithmetic modulo $p$ interacts with knot orientation reversal in this case. In case (\ref{item:case3}) we use the fact that $\tau$ is a smooth slice genus lower bound together with a covering link argument. The argument for case (\ref{item:case4}) requires a completely different method. This case is exceptional because, unique among lens spaces $L(2n,q)$, the preimage of a knot in $L(2,1)$ representing the class of order~$2$ is a $1$--component link in $S^3$, i.e.\ a knot, so the covering link argument of case~(\ref{item:case3}) does not work, and a different approach is required.

\subsection{Proof of Theorem~\ref{thm:B} when $x$ is not order 0 and $Y\neq \R P^3$}

\begin{proposition}\label{prop:lens_space_generic}Let $x\in[S^1,L(p,q)]$ have order different from 0 and 2. Then there exist topologically concordant knots $K$ and $K'$ representing $x$ which are distinct in smooth almost concordance.
\end{proposition}

\begin{proof}
Fix an identification $\pi_1(L(p,q))\cong \Z/p\Z$ by choosing an oriented meridian~$\eta$ to the $-p/q$ surgery curve in the surgery diagram defining $L(p,q)$. Under this identification~$[x]=\ell$ for some $0<\ell<p$, $2\ell\neq p$. For an integer $a$ and a knot~$J\subset S^3$, consider the curve~$K(J,a)\subset L(p,q)$ depicted in Figure~\ref{fig:Kplus}. By Lemma~\ref{lem:coveringlink-lens-spaces}, the preimage of~$K(J,a)$ in the universal cover is given by the second picture in Figure~\ref{fig:Kplus}. Choose a single component~$\wt{K}(J,a)$ of the preimage.

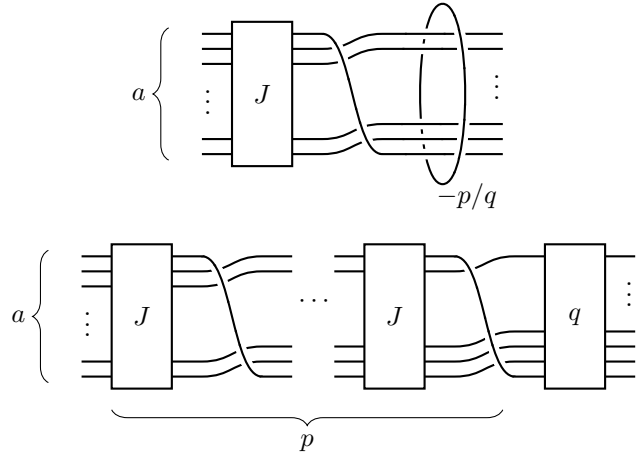
\begin{figure}[h]

\begin{tikzpicture}[scale=.8]

\begin{scope}[shift={(2,3.7)}]\begin{knot}[
clip radius=7pt,
clip width=5,
]

\strand[black,thick] (2,2) ..  controls +(0,0) and +(-0,0) .. (4,2) .. controls +(0.5,0) and +(-0.5,0) .. (5,0) .. controls +(0.5,0) and +(-0.5,0) .. (7,0);

\strand[black,thick] (2,1.75) ..  controls +(0,0) and +(-0,0) .. (4,1.75) .. controls +(0.5,0) and +(-0.5,0) .. (5,2) .. controls +(0.5,0) and +(-0.5,0) .. (7,2);

\strand[black,thick] (2,1.5) ..  controls +(0,0) and +(-0,0) .. (4,1.5) .. controls +(0.5,0) and +(-0.5,0) .. (5,1.75) .. controls +(0.5,0) and +(-0.5,0) .. (7,1.75);

\strand[black,thick] (2,0.25) ..  controls +(0,0) and +(-0,0) .. (4,0.25) .. controls +(0.5,0) and +(-0.5,0) .. (5,0.5) .. controls +(0.5,0) and +(-0.5,0) .. (7,0.5);

\strand[black,thick] (2,0) ..  controls +(0,0) and +(-0,0) ..  (4,0) .. controls +(0.5,0) and +(-0.5,0) .. (5,0.25) .. controls +(0.5,0) and +(-0.5,0) .. (7,0.25);

\strand[black,thick] (6,2.5) .. controls +(0.5,0) and +(0.5,0) .. (6,-0.5) .. controls +(-0.5,0) and +(-0.5,0) .. (6,2.5);

\flipcrossings{13,11,9,7,5}

\end{knot}

\node at (6.4,-0.7) {$-p/q$};
\node at (2.1,1) {$\vdots$};
\node at (6.9,1.25) {$\vdots$};

\draw [thick, draw=black, fill=white]
       (2.5, -0.2) -- (3.5,-0.2) -- (3.5,2.2) -- (2.5,2.2) -- cycle;
\node[rotate=0] at (3,1) {$J$};

\draw [decorate,decoration={brace,amplitude=6pt,raise=8pt}] (1.8,-0.1) -- (1.8,2.1);
\node at (0.95,1) {$a$};

\end{scope}

\begin{scope}[shift={(0,0)}]\begin{knot}[
clip radius=7pt,
clip width=5,
]

\strand[black,thick] (2,2) ..  controls +(0,0) and +(-0,0) .. (4,2) .. controls +(0.5,0) and +(-0.5,0) .. (5,0) .. controls +(0.5,0) and +(-0.5,0) .. (5.5,0);

\strand[black,thick] (2,1.75) ..  controls +(0,0) and +(-0,0) .. (4,1.75) .. controls +(0.5,0) and +(-0.5,0) .. (5,2) .. controls +(0.5,0) and +(-0.5,0) .. (5.5,2);

\strand[black,thick] (2,1.5) ..  controls +(0,0) and +(-0,0) .. (4,1.5) .. controls +(0.5,0) and +(-0.5,0) .. (5,1.75) .. controls +(0.5,0) and +(-0.5,0) .. (5.5,1.75);

\strand[black,thick] (2,0.25) ..  controls +(0,0) and +(-0,0) .. (4,0.25) .. controls +(0.5,0) and +(-0.5,0) .. (5,0.5) .. controls +(0.5,0) and +(-0.5,0) .. (5.5,0.5);

\strand[black,thick] (2,0) ..  controls +(0,0) and +(-0,0) ..  (4,0) .. controls +(0.5,0) and +(-0.5,0) .. (5,0.25) .. controls +(0.5,0) and +(-0.5,0) .. (5.5,0.25);

\flipcrossings{13,11,9,7,5}

\end{knot}

\node at (2.1,1) {$\vdots$};
\node at (5.9,1.25) {$\cdots$};

\draw [thick, draw=black, fill=white]
       (2.5, -0.2) -- (3.5,-0.2) -- (3.5,2.2) -- (2.5,2.2) -- cycle;
\node[rotate=0] at (3,1) {$J$};

\draw [decorate,decoration={brace,amplitude=6pt,raise=8pt}] (1.8,-0.1) -- (1.8,2.1);
\node at (0.95,1) {$a$};

\end{scope}

\begin{scope}[shift={(4.2,0)}]\begin{knot}[
clip radius=7pt,
clip width=5,
]

\strand[black,thick] (2,2) ..  controls +(0,0) and +(-0,0) .. (4,2) .. controls +(0.5,0) and +(-0.5,0) .. (5,0) .. controls +(0.5,0) and +(-0.5,0) .. (7,0);

\strand[black,thick] (2,1.75) ..  controls +(0,0) and +(-0,0) .. (4,1.75) .. controls +(0.5,0) and +(-0.5,0) .. (5,2) .. controls +(0.5,0) and +(-0.5,0) .. (7,2);

\strand[black,thick] (2,0.5) ..  controls +(0,0) and +(-0,0) .. (4,0.5) .. controls +(0.5,0) and +(-0.5,0) .. (5,0.75) .. controls +(0.5,0) and +(-0.5,0) .. (7,0.75);

\strand[black,thick] (2,0.25) ..  controls +(0,0) and +(-0,0) .. (4,0.25) .. controls +(0.5,0) and +(-0.5,0) .. (5,0.5) .. controls +(0.5,0) and +(-0.5,0) .. (7,0.5);

\strand[black,thick] (2,0) ..  controls +(0,0) and +(-0,0) ..  (4,0) .. controls +(0.5,0) and +(-0.5,0) .. (5,0.25) .. controls +(0.5,0) and +(-0.5,0) .. (7,0.25);

\flipcrossings{13,11,9,7,5}

\end{knot}

\node at (6.9,1.5) {$\vdots$};

\draw [thick, draw=black, fill=white]
       (2.5, -0.2) -- (3.5,-0.2) -- (3.5,2.2) -- (2.5,2.2) -- cycle;
\node[rotate=0] at (3,1) {$J$};

\draw [thick, draw=black, fill=white]
       (5.5, -0.2) -- (6.5,-0.2) -- (6.5,2.2) -- (5.5,2.2) -- cycle;
\node[rotate=0] at (6,1) {$q$};

\end{scope}

\draw [decorate,decoration={brace,amplitude=6pt,raise=8pt}]  (9,-0.2) -- (2.5,-0.2);
\node at (5.75,-1.1) {$p$};

\end{tikzpicture}
\caption{The curve $K(J,a)$ in $L(p,q)$ and its preimage in $S^3$.}
\label{fig:Kplus}
\end{figure}
Orienting $K(J,a)$ to agree with the meridian $\eta$, and setting~$a=\ell$, we obtain a knot $K_+(J,\ell)$ representing $x$, for any $J$. Orienting the curve to disagree with the meridian $\eta$ and setting~$a=p-\ell$, we obtain a knot $K_-(J,p-\ell)$ representing~$x$, for any $J$. Write $-D=\Wh^-(-T_{2,3})$, $U\subset S^3$ for the unknot and define $d:=\gcd(\ell,p)=\gcd(p-\ell,p)$.

Define four knots in $L(p,q)$:\[\text{$K_+=K_+(U,\ell)$, $K'_+=K_+(-D,\ell)$, $K_-=K_-(U,p-\ell)$, $K_-'=K_-(-D,p-\ell)$.}\]
As $-D$ is topologically null concordant, there are topological concordances~$K_+\sim K_+'$ and~$K_-\sim K_-'$.

Suppose that $K_+$ is smoothly almost concordant to $K_+'$. Comparing with Figure~\ref{fig:Kplus} we see  that the covering links of $K_+$ and $K_+'$ each have $\gcd(\ell,p+\ell q)$ components, which is calculated to be $\gcd(\ell,p+\ell q)=\gcd(\ell, p)=d$ components. Restricting to a single component of each of the covering links, and recalling the discussion in Section~\ref{subsec:coveringlinks}, there exists some $\widehat{J}\subset S^3$ such that there is a smooth concordance between the knots
\[(p(-D))_{\frac{\ell}{d},\frac{p+\ell q}{d}}\#\left(\tmfrac{p}{d}\right)\widehat{J}\quad\text{and}\quad T_{\frac{\ell}{d},\frac{p+\ell q}{d}},\]
where $d:=\gcd(\ell,p)$. We now wish to use Van Cott's cabling formulae~\ref{prop:vancottcable}. We have already seen in Example~\ref{exm:SmoothFamily} that $\tau(\Wh^+(T_{2,3}))=1$. Recall as well that $\Wh^{-}(-J)=\Wh^+(J)$ for any knot $J\subset S^3$, so that $\Wh^{-}(-T_{2,3})=-\Wh^+(T_{2,3})$. As $\tau$ is a homomorphism from the knot concordance group, we conclude~$\tau(-D)=-1$. There is a genus~1 Seifert surface for~$\Wh^{-}(-T_{2,3})$, so~$\tau(-D)=-g(-D)$. We may thus apply Van Cott's cabling formulae, together with the Ozsv{\'a}th-Szab{\'o} torus knot formula, to calculate:
\[\begin{array}{*2{>{\displaystyle}l}}
&\tau\left((p(-D))_{\frac{\ell}{d},\frac{p+\ell q}{d}}\#\left(\tmfrac{p}{d}\right)\widehat{J}\right)= \tau\left(T_{\frac{\ell}{d},\frac{p+\ell q}{d}}\right)\\
\implies&\frac{\ell}{d}\cdot\tau(p(-D))+\frac{\left(\tmfrac{\ell}{d}-1\right)\left(\tmfrac{p+\ell q}{d}+1\right)}{2}+\tau\left(\frac{p}{d}\cdot \widehat{J}\right)=\frac{\left(\tmfrac{\ell}{d}-1\right)\left(\tmfrac{p+\ell q}{d}-1\right)}{2}\\
&\\
\implies&\frac{\ell p}{d}\cdot\tau(-D)+\frac{p}{d}\cdot\tau(\widehat{J})=1-\frac{\ell}{d}\\
&\\
\implies& 1\equiv \frac{\ell}{d}\quad\text{mod $\frac{p}{d}$}.\end{array}\]
Since $0< \frac{\ell}{d}< \frac{p}{d}$, this final statement is equivalent to the statement that $\ell=d$. Recalling that $d=\gcd(\ell, p)$, this is equivalent to the statement that $\ell$ divides $p$.

Suppose now that we also have that $K_-$ is smoothly almost concordant to $K_-'$. By exactly the same argument, we obtain the statement that $p-\ell$ divides $p$.

As $0<\ell<p$, at least one of $\ell$ and $p-\ell$ is greater than or equal to $p/2$. Further, since they both divide $p$, it follows that at least one of them is $p/2$, which implies that $\ell=p/2$. But this contradicts our initial hypotheses, as $x$ does not have order 2. Hence either $K_+$ is not smoothly concordant to $K_+'$, or $K_-$ is not smoothly concordant to $K_-'$. Either way, we obtain a pair of knots as claimed in the statement of the proposition.
\end{proof}

\begin{proposition}\label{prop:L2nqnotRP3}Suppose $n>1$ and let $x\in[S^1,L(2n,q)]$ be the element of order~$2$. Then there exist topologically concordant knots $K$ and $K'$ in $L(2n,q)$ representing $x$ which are distinct in smooth almost concordance.
\end{proposition}

\begin{proof}

For a knot $J\subset S^3$ define a knot $K_J$ in $L(2n,q)$ as in Figure~\ref{fig:definingfig}.
\begin{figure}[h]

\begin{tikzpicture}[scale=1.14]

\begin{scope}[shift={(0,0)}]\begin{knot}[
clip radius=7pt,
clip width=5,
]

\strand[black,thick] (2,2) ..  controls +(0,0) and +(-0,0) .. (4,2) .. controls +(0.5,0) and +(-0.5,0) .. (5,0) .. controls +(0.5,0) and +(-0.5,0) .. (7,0);

\strand[black,thick] (2,1.6) ..  controls +(0,0) and +(-0,0) .. (4,1.6) .. controls +(0.5,0) and +(-0.5,0) .. (5,2) .. controls +(0.5,0) and +(-0.5,0) .. (7,2);

\strand[black,thick] (2,1.2) ..  controls +(0,0) and +(-0,0) .. (4,1.2) .. controls +(0.5,0) and +(-0.5,0) .. (5,1.6) .. controls +(0.5,0) and +(-0.5,0) .. (7,1.6);

\strand[black,thick] (2,0.4) ..  controls +(0,0) and +(-0,0) .. (4,0.4) .. controls +(0.5,0) and +(-0.5,0) .. (5,0.8) .. controls +(0.5,0) and +(-0.5,0) .. (7,0.8);

\strand[black,thick] (2,0) ..  controls +(0,0) and +(-0,0) ..  (4,0) .. controls +(0.5,0) and +(-0.5,0) .. (5,0.4) .. controls +(0.5,0) and +(-0.5,0) .. (7,0.4);

\strand[black,thick] (6,2.5) .. controls +(0.5,0) and +(0.5,0) .. (6,-0.5) .. controls +(-0.5,0) and +(-0.5,0) .. (6,2.5);

\flipcrossings{13,11,9,7,5}

\end{knot}

\node at (6.4,-0.7) {$-2n/q$};
\node at (2.1,0.9) {$\vdots$};
\node at (6.9,1.25) {$\vdots$};

\draw [thick, draw=black, fill=white]
       (2.5, 1.4) -- (3.5,1.4) -- (3.5,2.2) -- (2.5,2.2) -- cycle;
\node[rotate=0] at (3,1.8) {$W(J)$};

\draw [decorate,decoration={brace,amplitude=6pt,raise=8pt}] (2,-0.1) -- (2,2.1);
\node at (1.4,1) {$n$};

\end{scope}

\begin{scope}[xscale=1.25, yscale=1.25, shift={(6.2,1.2)}]\begin{knot}[
clip radius=7pt,
clip width=5,
]

\strand [black, thick] (1.6,0.6)
to [out=right, in=left] (2.2,0.6)
to [out=right, in=left, looseness=1] (2.6,-0.65)
to [out=right, in=left, looseness=1] (3,0.6)
to [out=right, in=left] (3.6, 0.6);

\strand [black, thick] (1.6,-0.45)
to [out=right, in=left] (2.15, -0.45)
to [out=right, in=left] (3.05, -0.45)
to [out=right, in=left] (3.6, -0.45);

\flipcrossings{1}
\end{knot}

\draw [thick, draw=black, fill=white]
       (2.3, -0.25) -- (2.9,-0.25) -- (2.9,0.35) -- (2.3,0.35) -- cycle;
\node[rotate=90] at (2.6,0.05) {$J$};

\draw [thick, draw=black]
       (1.6, -0.8) -- (3.6,-0.8) -- (3.6,0.8) -- (1.6,0.8) -- cycle;

\draw [thick, draw=black, fill=white]
       (0, -0.4) -- (1,-0.4) -- (1,0.4) -- (0,0.4) -- cycle;
\node[rotate=0] at (0.5,0) {$W(J)$};

\node[rotate=0] at (1.3,0) {$:=$};

\end{scope}

\begin{scope}[shift={(8.1,-0.3)}]
\draw [thick, draw=black, fill=white]
       (0, -0.4) -- (1,-0.4) -- (1,0.4) -- (0,0.4) -- cycle;
\node[rotate=0] at (0.5,0) {$W(J)$};
\draw[black,thick,->] (0,-0.25)
to [out=left, in=down] (-0.25,0)
to [out=up, in=left] (0,0.25);
\draw[black,thick,->](1,0.25)
to [out=right, in=up] (1.25,0)
to [out=down, in=right] (1,-0.25);
\node[rotate=0] at (2.2,0) {$=\Wh^+(J)$};
\end{scope}

\end{tikzpicture}
\caption{The knot $K_J\subset L(2n,q)$ representing the element $x$ of order 2. In the left-hand diagram, the strands are oriented from left to right.}
\label{fig:definingfig}
\end{figure}
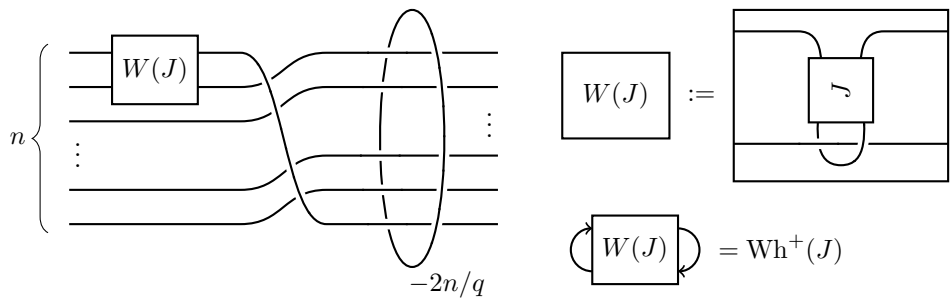
This knot winds algebraically $n$ times around the surgery curve for $L(2n,q)$ and so represents the element of order 2 in $[S^1,L(2n,q)]$. Note that closing off the $W(J)$--box, as indicated in Figure~\ref{fig:definingfig}, results in the positive Whitehead double of $J$.

Write $\widetilde{K}_J$ for the link in $S^3$ obtained as the covering link of $K_J$ under the universal cover. Applying Lemma~\ref{lem:coveringlink-lens-spaces} to $K_J$, we see there are $\gcd(n,2n)=n$ components of $\widetilde{K}_J$. Each component of $\widetilde{K}_J$ intersects any given $W(J)$--box in either~$0$ or~$2$ points (although a given component might hit multiple boxes), showing that these boxes do not affect the knotting of a single component of~$\widetilde{K}_J$. Thus the components of $\wt{K}_J$ are torus knots $T(1,2+q)$, i.e.\ are unknotted. Now we choose a two component sublink $L_J\subset S^3$ of $\widetilde{K}_J$, consisting of two adjacent strands. The possibilities are depicted in Figure~\ref{fig:2component}.
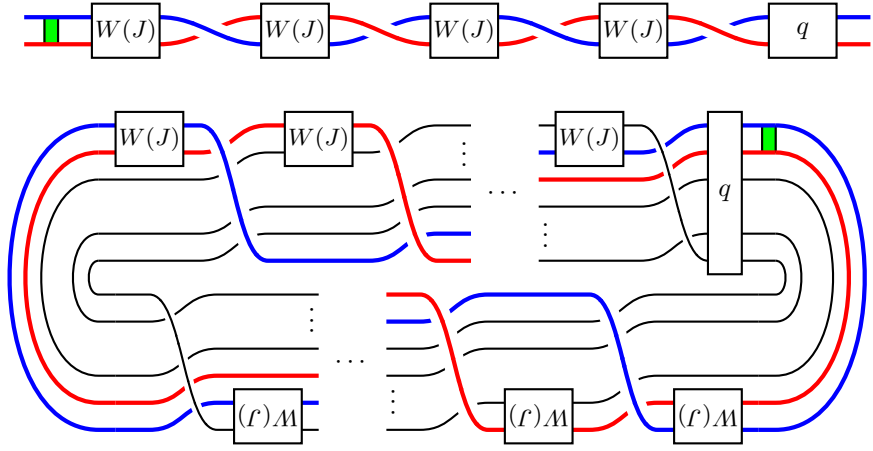
\begin{figure}
\begin{tikzpicture}[scale=0.9]

\begin{scope}[shift={(0,0)}]

\begin{knot}[
clip radius=7pt,
clip width=5,
]

\strand[blue,ultra thick] (2.5,2) ..  controls +(0,0) and +(-0,0) .. (4,2) .. controls +(0.5,0) and +(-0.5,0) .. (5,0)  ..  controls +(0,0) and +(-0,0) ..  (6.5,0) .. controls +(0.5,0) and +(-0.5,0) .. (7.5,0.4) .. controls +(0.5,0) and +(-0.5,0) .. (8,0.4);

\strand[red, ultra thick] (2.5,1.6) ..  controls +(0,0) and +(-0,0) .. (4,1.6) .. controls +(0.5,0) and +(-0.5,0) .. (5,2)  ..  controls +(0,0) and +(-0,0) .. (6.5,2)  .. controls +(0.5,0) and +(-0.5,0) .. (7.5,0)  .. controls +(0.5,0) and +(-0.5,0) .. (8,0);

\strand[black,thick] (2.5,1.2) ..  controls +(0,0) and +(-0,0) .. (4,1.2) .. controls +(0.5,0) and +(-0.5,0) .. (5,1.6) ..  controls +(0,0) and +(-0,0) .. (6.5,1.6) .. controls +(0.5,0) and +(-0.5,0) .. (7.5,2) .. controls +(0.5,0) and +(-0.5,0) .. (8,2);

\strand[black,thick] (2.5,0.4) ..  controls +(0,0) and +(-0,0) .. (4,0.4) .. controls +(0.5,0) and +(-0.5,0) .. (5,0.8) ..  controls +(0,0) and +(-0,0) .. (6.5,0.8) .. controls +(0.5,0) and +(-0.5,0) .. (7.5,1.2) .. controls +(0.5,0) and +(-0.5,0) .. (8,1.2);

\strand[black,thick] (2.5,0) ..  controls +(0,0) and +(-0,0) ..  (4,0) .. controls +(0.5,0) and +(-0.5,0) .. (5,0.4) ..  controls +(0,0) and +(-0,0) .. (6.5,0.4) .. controls +(0.5,0) and +(-0.5,0) .. (7.5,0.8) .. controls +(0.5,0) and +(-0.5,0) .. (8,0.8);

\flipcrossings{2}

\end{knot}

\draw [thick, draw=black, fill=white]
       (2.75, 1.4) -- (3.75,1.4) -- (3.75,2.2) -- (2.75,2.2) -- cycle;
\node[rotate=0] at (3.25,1.8) {\small{$W(J)$}};

\draw [thick, draw=black, fill=white]
       (5.25, 1.4) -- (6.25,1.4) -- (6.25,2.2) -- (5.25,2.2) -- cycle;
\node[rotate=0] at (5.75,1.8) {\small{$W(J)$}};

\node at (7.9,1.7) {$\vdots$};

\end{scope}

\begin{scope}[shift={(6.5,0)}]

\draw [thick, draw=green, fill=green]
       (5.8, 1.6) -- (6,1.6) -- (6,2) -- (5.8,2) -- cycle;
\draw [thick, black]
       (5.8, 1.6) -- (5.8, 2);
\draw [thick, black]
       (6, 1.6) -- (6, 2);

\begin{knot}[
clip radius=7pt,
clip width=5,
]

\strand[black,thick] (2.5,2) ..  controls +(0,0) and +(-0,0) .. (4,2) .. controls +(0.5,0) and +(-0.5,0) .. (5,0) .. controls +(0.5,0) and +(-0.5,0) .. (6,0);

\strand[blue,ultra thick] (2.5,1.6) ..  controls +(0,0) and +(-0,0)  .. (4,1.6) .. controls +(0.5,0) and +(-0.5,0) .. (5,2) .. controls +(0.5,0) and +(-0.5,0) .. (6,2)  ;

\strand[red, ultra thick] (2.5,1.2) ..  controls +(0,0) and +(-0,0) .. (4,1.2)  .. controls +(0.5,0) and +(-0.5,0) .. (5,1.6) .. controls +(0.5,0) and +(-0.5,0) .. (6,1.6) ;

\strand[black,thick] (2.5,0.8) ..  controls +(0,0) and +(-0,0) .. (4,0.8) .. controls +(0.5,0) and +(-0.5,0) .. (5,1.2) .. controls +(0.5,0) and +(-0.5,0) .. (6,1.2) ;

\strand[black,thick] (2.5,0) ..  controls +(0,0) and +(-0,0) ..  (4,0) .. controls +(0.5,0) and +(-0.5,0) .. (5,0.4)  .. controls +(0.5,0) and +(-0.5,0) .. (6,0.4) ;

\flipcrossings{}

\end{knot}

\draw [thick, draw=black, fill=white]
       (5, -0.2) -- (5.5,-0.2) -- (5.5,2.2) -- (5,2.2) -- cycle;
\node[rotate=0] at (5.25,1) {$q$};

\draw [thick, draw=black, fill=white]
       (2.75, 1.4) -- (3.75,1.4) -- (3.75,2.2) -- (2.75,2.2) -- cycle;
\node[rotate=0] at (3.25,1.8) {\small{$W(J)$}};

\node at (2.6,0.5) {$\vdots$};

\node at (2,1) {$\cdots$};

\end{scope}

\begin{scope}[shift={(14.75,-0.5)}, rotate=180]

\begin{scope}[shift={(0,0)}]\begin{knot}[
clip radius=7pt,
clip width=5,
]

\strand[blue,ultra thick] (2.5,2) ..  controls +(0,0) and +(-0,0) .. (4,2) .. controls +(0.5,0) and +(-0.5,0) .. (5,0)  ..  controls +(0,0) and +(-0,0) ..  (6.5,0) .. controls +(0.5,0) and +(-0.5,0) .. (7.5,0.4) .. controls +(0.5,0) and +(-0.5,0) .. (8,0.4);

\strand[red,ultra thick] (2.5,1.6) ..  controls +(0,0) and +(-0,0) .. (4,1.6) .. controls +(0.5,0) and +(-0.5,0) .. (5,2)  ..  controls +(0,0) and +(-0,0) .. (6.5,2)  .. controls +(0.5,0) and +(-0.5,0) .. (7.5,0)  .. controls +(0.5,0) and +(-0.5,0) .. (8,0);

\strand[black,thick] (2.5,1.2) ..  controls +(0,0) and +(-0,0) .. (4,1.2) .. controls +(0.5,0) and +(-0.5,0) .. (5,1.6) ..  controls +(0,0) and +(-0,0) .. (6.5,1.6) .. controls +(0.5,0) and +(-0.5,0) .. (7.5,2) .. controls +(0.5,0) and +(-0.5,0) .. (8,2);

\strand[black,thick] (2.5,0.4) ..  controls +(0,0) and +(-0,0) .. (4,0.4) .. controls +(0.5,0) and +(-0.5,0) .. (5,0.8) ..  controls +(0,0) and +(-0,0) .. (6.5,0.8) .. controls +(0.5,0) and +(-0.5,0) .. (7.5,1.2) .. controls +(0.5,0) and +(-0.5,0) .. (8,1.2);

\strand[black,thick] (2.5,0) ..  controls +(0,0) and +(-0,0) ..  (4,0) .. controls +(0.5,0) and +(-0.5,0) .. (5,0.4) ..  controls +(0,0) and +(-0,0) .. (6.5,0.4) .. controls +(0.5,0) and +(-0.5,0) .. (7.5,0.8) .. controls +(0.5,0) and +(-0.5,0) .. (8,0.8);

\flipcrossings{2}

\end{knot}

\draw [thick, draw=black, fill=white]
       (2.75, 1.4) -- (3.75,1.4) -- (3.75,2.2) -- (2.75,2.2) -- cycle;
\node[rotate=180] at (3.25,1.8) {\small{$W(J)$}};

\draw [thick, draw=black, fill=white]
       (5.25, 1.4) -- (6.25,1.4) -- (6.25,2.2) -- (5.25,2.2) -- cycle;
\node[rotate=180] at (5.75,1.8) {\small{$W(J)$}};

\node[rotate=180] at (7.9,1.7) {$\vdots$};

\end{scope}

\begin{scope}[shift={(6.5,0)}]\begin{knot}[
clip radius=7pt,
clip width=5,
]

\strand[black,thick] (2.5,2) ..  controls +(0,0) and +(-0,0) .. (4,2) .. controls +(0.5,0) and +(-0.5,0) .. (5,0) .. controls +(0.5,0) and +(-0.5,0) .. (5.5,0);

\strand[blue, ultra thick] (2.5,1.6) ..  controls +(0,0) and +(-0,0)  .. (4,1.6) .. controls +(0.5,0) and +(-0.5,0) .. (5,2) .. controls +(0.5,0) and +(-0.5,0) .. (5.5,2)  ;

\strand[red, ultra thick] (2.5,1.2) ..  controls +(0,0) and +(-0,0) .. (4,1.2)  .. controls +(0.5,0) and +(-0.5,0) .. (5,1.6) .. controls +(0.5,0) and +(-0.5,0) .. (5.5,1.6) ;

\strand[black,thick] (2.5,0.8) ..  controls +(0,0) and +(-0,0) .. (4,0.8) .. controls +(0.5,0) and +(-0.5,0) .. (5,1.2) .. controls +(0.5,0) and +(-0.5,0) .. (5.5,1.2) ;

\strand[black,thick] (2.5,0) ..  controls +(0,0) and +(-0,0) ..  (4,0) .. controls +(0.5,0) and +(-0.5,0) .. (5,0.4)  .. controls +(0.5,0) and +(-0.5,0) .. (5.5,0.4) ;

\flipcrossings{}

\end{knot}

\draw [thick, draw=black, fill=white]
       (2.75, 1.4) -- (3.75,1.4) -- (3.75,2.2) -- (2.75,2.2) -- cycle;
\node[rotate=180] at (3.25,1.8) {\small{$W(J)$}};

\node[rotate=180] at (2.6,0.5) {$\vdots$};

\node at (2,1) {$\cdots$};

\end{scope}
\end{scope}

\draw[blue,ultra thick] (2.5,2)
to [out=left, in=left] (2.5,-2.5)
to [out=right, in=left] (2.75,-2.5);

\draw[red,ultra thick] (2.5,1.6)
to [out=left, in=left] (2.5,-2.1)
to [out=right, in=left] (2.75,-2.1);

\draw[black, thick] (2.5,1.2)
to [out=left, in=left] (2.5,-1.7)
to [out=right, in=left] (2.75,-1.7);

\draw[black, thick] (2.5,0.4)
to [out=left, in=left] (2.5,-0.9)
to [out=right, in=left] (2.75,-0.9);

\draw[black, thick] (2.5,0)
to [out=left, in=left] (2.5,-0.5)
to [out=right, in=left] (2.75,-0.5);

\draw[blue,ultra thick] (12.5,2)
to [out=right, in=right] (12.5,-2.5)
to [out=left, in=right] (12.25,-2.5);

\draw[red,ultra thick] (12.5,1.6)
to [out=right, in=right] (12.5,-2.1)
to [out=left, in=right] (12.25,-2.1);

\draw[black, thick] (12.5,1.2)
to [out=right, in=right] (12.5,-1.7)
to [out=left, in=right] (12.25,-1.7);

\draw[black, thick] (12.5,0.4)
to [out=right, in=right] (12.5,-0.9)
to [out=left, in=right] (12.25,-0.9);

\draw[black, thick] (12.5,0)
to [out=right, in=right] (12.5,-0.5)
to [out=left, in=right] (12.25,-0.5);

\begin{scope}[shift={(2.4,3)}]

\draw [thick, draw=green, fill=green]
       (-0.7, 0.2) -- (-0.5,0.2) -- (-0.5,0.6) -- (-0.7,0.6) -- cycle;
\draw [thick, black]
       (-0.7, 0.2) -- (-0.7, 0.6);
\draw [thick, black]
       (-0.5, 0.2) -- (-0.5, 0.6);

\begin{knot}[
clip radius=7pt,
clip width=5,
]

\strand [blue,ultra thick] (-1,0.6)
to [out=right, in=left] (0,0.6)
to [out=right, in=left] (1,0.6)
to [out=right, in=left] (2.5,0.2)
to [out=right, in=left] (3.5,0.2)
to [out=right, in=left] (5,0.6)
to [out=right, in=left] (6,0.6)
to [out=right, in=left] (7.5,0.2)
to [out=right, in=left] (8.5,0.2)
to [out=right, in=left] (10,0.6)
to [out=right, in=left] (11,0.6)
to [out=right, in=left] (11.5,0.6);

\strand [red,ultra thick] (-1,0.2)
to [out=right, in=left] (0,0.2)
to [out=right, in=left] (1,0.2)
to [out=right, in=left] (2.5,0.6)
to [out=right, in=left] (3.5,0.6)
to [out=right, in=left] (5,0.2)
to [out=right, in=left] (6,0.2)
to [out=right, in=left] (7.5,0.6)
to [out=right, in=left] (8.5,0.6)
to [out=right, in=left] (10,0.2)
to [out=right, in=left] (11,0.2)
to [out=right, in=left] (11.5,0.2);

\flipcrossings{2,4}

\end{knot}

\draw [thick, draw=black, fill=white]
       (0, 0) -- (1,0) -- (1,0.8) -- (0,0.8) -- cycle;
\node[rotate=0] at (0.5,0.4) {\small{$W(J)$}};

\draw [thick, draw=black, fill=white]
       (2.5, 0) -- (3.5,0) -- (3.5,0.8) -- (2.5,0.8) -- cycle;
\node[rotate=0] at (3,0.4) {\small{$W(J)$}};

\draw [thick, draw=black, fill=white]
       (5, 0) -- (6,0) -- (6,0.8) -- (5,0.8) -- cycle;
\node[rotate=0] at (5.5,0.4) {\small{$W(J)$}};

\draw [thick, draw=black, fill=white]
       (7.5, 0) -- (8.5,0) -- (8.5,0.8) -- (7.5,0.8) -- cycle;
\node[rotate=0] at (8,0.4) {\small{$W(J)$}};

\draw [thick, draw=black, fill=white]
       (10, 0) -- (11,0) -- (11,0.8) -- (10,0.8) -- cycle;
\node[rotate=0] at (10.5,0.4) {$q$};

\end{scope}

\end{tikzpicture}
\caption{2--component sublinks $L_J$ of $\widetilde{K}_J$ when $n=2$ and when~$n>2$. Each case has a single band move, that can only be performed after orientation reversal on one strand.  Before the band move, in the top figure the orientations run from left to right, while in the bottom figure the orientations run clockwise.}
\label{fig:2component}
\end{figure}
We wish to perform band moves as indicated in Figure~\ref{fig:2component}. In order to do so, first reverse the orientation on one of the link components touched by the band. Call the resulting link $\widehat{L}_J$. Now perform the indicated band move. When $n=2$, the band move determines a smoothly embedded surface~$\Sigma(J)\subset S^3\times [0,1]$ between $\widehat{L}_J$ and the connected sum~$2\Wh^+(J)\#2(\Wh^+(J))^r$. When $n>2$, the band move determines an embedded smooth surface from $\widehat{L}_J$ to $\Wh^+(J)\#\Wh^+(J)$. In both cases, $\Sigma(J)$ is a pair of trousers.

Suppose $J\subset S^3$ is such that $K_J$ and $K_{U}$ are smoothly almost concordant where~$U\subset S^3$ is the unknot. As the components of the covering links are unknotted, we may apply Proposition~\ref{prop:almostimpliesconc} to conclude that the covering links $\widetilde{K}_J$ and $\widetilde{K}_{U}$ are smoothly concordant. Form the link $L_J$ and note that it is concordant to a 2--component sublink of $\widetilde{K}_{U}$. Any $2$--component sublink of $\widetilde{K}_{U}$ is a torus link~$T_{2,q+2}$. Form the link $\widehat{L}_J$ by reversing the orientation on one component of $L_J$. Reverse the orientation on the corresponding component of the concordance and on the corresponding component of $\widetilde{K}_{U}$ to obtain a concordance $A\subset S^3\times [0,1]$ from~$\widehat{L}_J$ to a $2$--component link $\widehat{U}$. Choose some band move on $\widehat{U}$ which results in the unknot and write $\Sigma'\subset S^3\times[0,1]$ for the surface which is the trace of that band move. Concatenate smoothly embedded surfaces as follows:\[\Sigma:=\Sigma(J)\cup_{\widehat{L}_J}A\cup_{\widehat{U}}\Sigma'\subset S^3\times[0,1],\]drawn schematically in Figure~\ref{fig:surface}.
\begin{figure}

\begin{tikzpicture}[scale=1.7]

\draw[black, thick] (-2,0.3)
to [out=right, in=left] (-1,0.9)
to [out=right, in=left] (0,0.9)
to [out=right, in=left] (1,0.9)
to [out=right, in=left] (2,0.3);

\draw [thick, black] (2,-0.3)
to [out=left, in=right] (1,-0.9)
to [out=left, in=right] (0,-0.9)
to [out=left, in=right] (-1,-0.9)
to [out=left, in=right] (-2,-0.3);

\draw[black, thick] (-1,0.3)
to [out=right, in=left] (1,0.3)
to [out=right, in=right, looseness=1.2] (1, -0.3)
to [out=left, in=right] (-1, -0.3)
to [out=left, in=left, looseness=1.2] (-1, 0.3);

\draw[black, thick, dashed] (-2,0.3)
to [out=right, in=right] (-2, -0.3);
\draw[black, thick] (-2,- 0.3)
to [out=left, in=left] (-2,0.3);

\draw[black, thick, dashed] (-1,0.9)
to [out=right, in=right] (-1, 0.3);
\draw[black, thick] (-1, 0.3)
to [out=left, in=left] (-1,0.9);

\draw[black, thick, dashed] (1,0.9)
to [out=right, in=right] (1, 0.3);
\draw[black, thick] (1, 0.3)
to [out=left, in=left] (1,0.9);

\draw[black, thick, dashed] (-1,-0.3)
to [out=right, in=right] (-1, -0.9);
\draw[black, thick] (-1, -0.9)
to [out=left, in=left] (-1,-0.3);

\draw[black, thick, dashed] (1,-0.3)
to [out=right, in=right] (1, -0.9);
\draw[black, thick] (1, -0.9)
to [out=left, in=left] (1,-0.3);

\draw[black, thick] (2,0.3)
to [out=right, in=right] (2, -0.3);
\draw[black, thick] (2,- 0.3)
to [out=left, in=left] (2,0.3);

\node[rotate=0] at (0,0) {$A$};
\node[rotate=0] at (-1.5,0) {\small{$\Sigma(J)$}};
\node[rotate=0] at (1.53,0) {\small{$\Sigma'$}};
\node[rotate=0] at (-1,-1.2) {$\widehat{L}_J$};
\node[rotate=0] at (1,-1.2) {$\widehat{U}$};
\draw [decorate,decoration={brace,amplitude=6pt,raise=8pt}] (2.5,1) -- (2.5,-1);
\node at (3,0) {$\Sigma$};

\draw[black, thick, ->] (0.15,0.05)
to [out=25, in=-45] (0.25,0.5);

\draw[black, thick, ->] (0.15,-0.05)
to [out=-25, in=45] (0.25,-0.5);

\end{tikzpicture}
\caption{The genus one surface $\Sigma$.}
\label{fig:surface}
\end{figure}
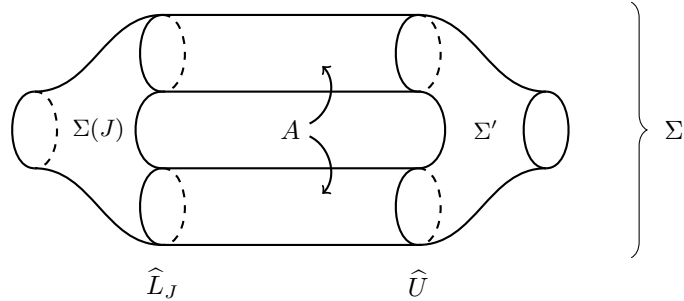
Cap off the unknotted end of $\Sigma$ by an unknotted~$D^2\subset D^4$. The capped off surface $\Sigma\cup D^2$ is a smoothly embedded genus 1 surface in the 4--ball that bounds
\[\begin{array}{rl}
2\Wh^+(J)\#2(\Wh^+(J))^r& \text{when $n=2$},\\
\Wh^+(J)\#\Wh^+(J)& \text{when $n>2$.}
\end{array}\]

Now we set $J=\Wh^+(T_{2,3})$. Recall the absolute value of the $\tau$ invariant is a lower bound for the smooth slice genus. We have already seen that $\tau(\Wh^+(T_{2,3}))=1$, so by Theorem~\ref{thm:hedden} we have $\tau(\Wh^+(J))=\tau(\Wh^+(\Wh^+(T_{2,3})))=1$. Hence $\tau(2\Wh^+(J)\#2(\Wh^+(J))^r)=4$ and $\tau(\Wh^+(J)\#\Wh^+(J))=2$. In each case, this contradicts the existence of a smoothly embedded genus 1 slicing surface. This contradiction shows that $K_J$ is not smoothly almost concordant to $K_{U}$. However, the positive Whitehead double of a knot is always topologically concordant to the unknot, so $K_J$ and $K_{U}$ are topologically concordant.
\end{proof}

\subsection{Proof of Theorem~\ref{thm:B} when $x$ is order 2 and $Y= \R P^3$}

We will make an argument using the Ozsv{\'a}th-Stipsicz-Szab{\'o} $\Upsilon$ invariant \cite{Ozsvath14} as the tool to distinguish the smooth almost concordance classes when $Y=\R P^3$ and $x$ is of order 2. We recall some features of this invariant. For a knot $J\subset S^3$ there is a certain piecewise linear function $\Upsilon_J\colon [0,2]\to \R$ such that $\Upsilon_J(0)=\Upsilon_J(2)=0$ and~$\Upsilon_J(t)=\Upsilon_J(2-t)$. The $\Upsilon$ invariants define a group homomorphism from $\mathcal{C}^{\Diff}$ to the group of $\R$--valued continuous functions on $[0,2]$. When the slope of $\Upsilon_J(t)$ is defined, it is an integer \cite[Proposition 1.4]{Ozsvath14}.

\begin{proposition}\label{prop:RP3}Let $x\in[S^1,L(2,1)]$ be the element of order 2. Then there exist topologically concordant knots $K$ and $K'$ representing $x$ which are distinct in smooth almost concordance.
\end{proposition}

\begin{proof}Let $J\subset S^3$ be a knot and consider the knot $K_J$ in $L(2,1)$, depicted in Figure~\ref{fig:exceptional}.
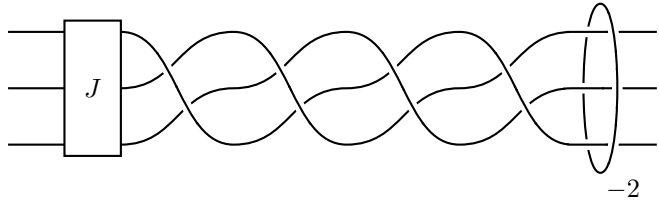
\begin{figure}

\begin{tikzpicture}[scale=0.75]
\begin{knot}[
clip radius=7pt,
clip width=5,
]

\strand[black,thick] (0,2) .. controls +(1,0) and +(-1,0) .. (2,0) .. controls +(1,0) and +(-1,0) .. (4,1) .. controls +(1,0) and +(-1,0) .. (6,2) .. controls +(1,0) and +(-1,0) .. (8,0) .. controls +(1,0) and +(-1,0) .. (9.5,0);

\strand[black,thick] (0,1) .. controls +(1,0) and +(-1,0) .. (2,2) .. controls +(1,0) and +(-1,0) .. (4,0) .. controls +(1,0) and +(-1,0) .. (6,1) .. controls +(1,0) and +(-1,0) .. (8,2) .. controls +(1,0) and +(-1,0) .. (9.5,2);

\strand[black,thick] (0,0) .. controls +(1,0) and +(-1,0) .. (2,1) .. controls +(1,0) and +(-1,0) .. (4,2) .. controls +(1,0) and +(-1,0) .. (6,0) .. controls +(1,0) and +(-1,0) .. (8,1) .. controls +(1,0) and +(-1,0) .. (9.5,1);

\strand[black,thick] (8.5,2.5) .. controls +(0.4,0) and +(0.4,0) .. (8.5,-0.5) .. controls +(-0.4,0) and +(-0.4,0) .. (8.5,2.5);

\flipcrossings{2, 5, 10, 11, 7, 13}
\end{knot}

\node at (8.9,-0.8) {$-2$};

\draw [thick, draw=black, fill=white]
       (-1,-0.2) -- (0,-0.2) -- (0,2.2) -- (-1,2.2) -- cycle;
\node at (-0.5,1) {$J$};

\draw[black, thick] (-1,0) to (-2,0);
\draw[black, thick] (-1,1) to (-2,1);
\draw[black, thick] (-1,2) to (-2,2);

\end{tikzpicture}
\caption{The knot $K_J$ in $L(2,1)$. The orientations run from left to right.}
\label{fig:exceptional}
\end{figure}
Since $K_J$ passes an odd number of times through the surgery curve it represents the class $x$. By Lemma~\ref{lem:coveringlink-lens-spaces}, the preimage of $K_J$ in the universal cover is the $(3,11)$ cable of $J\# J$ in $S^3$. Now write $D:=\Wh^+(T_{2,3})$ and $U$ for the unknot in $S^3$. Note that as $D$ is topologically concordant to $U$, $K_D$ is topologically concordant to $K_U$. We claim they are not smoothly almost concordant. Suppose, for a contradiction, that there exists a knot $\widehat{J}\subset S^3$ such that $K_D$ is smoothly concordant to $K_U\# \widehat{J}$. Then, lifting to the universal cover $S^3$, there is a smooth concordance between $(D\# D)_{3,11}$ and $T_{3,11}\#\widehat{J}\# \widehat{J}$. Thus there is a smooth concordance between $(D\# D)_{3,11}\#-T_{3,11}$ and $\widehat{J}\# \widehat{J}$. We prove below, in Proposition~\ref{prop:genwhiteheaddoubles}, that the slope of the $\Upsilon$ invariant of $(D\# D)_{3,11}\#-T_{3,11}$ is constantly 5 when $\frac{2}{5} < t < \frac{2}{3}$. In particular this is an odd number, whereas the slope of the $\Upsilon$ invariant of $\widehat{J}\# \widehat{J}$ must be even for all $t\in[0,2]$ where the slope is defined, as $\Upsilon$ is additive under connected sum. This contradiction shows that $K_D$ and $K_U$ are not smoothly almost concordant in $L(2,1)$.
\end{proof}

It remains to calculate the claimed $\Upsilon$ invariant slope that we just used. To do so, we will need some further ideas from Heegaard-Floer homology.

Write $\widehat{HF}$ to denote the `hat' version of Heegaard-Floer homology \cite{MR2065507}. A rational homology sphere $X$ is an \emph{$L$--space} if $\rk(\widehat{HF}(X))=|H_1(X;\Z)|$. A knot~$J\subset S^3$ is called an \emph{$L$--space knot} if there exists $n\in\Z_{>0}$ such that $n$--framed surgery on $S^3$ along $J$ returns an $L$--space. By \cite[Proposition 3.2]{Mos71}, for any $r,s>0$ coprime, there exists $n\in\Z_{>0}$ such that $n$--framed surgery on $S^3$ along the torus knot $T_{r,s}$ is a lens space. By \cite[Proposition 2.3]{MR2168576}, all lens spaces are $L$--spaces, so any such torus knot is an $L$--space knot.

In \cite{MR3523259}, Hom and Wu define a non-negative integer valued smooth concordance invariant $\nu^+$ for knots in $S^3$. Following M.\ H.\ Kim and K.\ Park \cite{Kim:2016uq}, we say that two knots $J,J'\subset S^3$ are \emph{$\nu^+$--equivalent} if
\[\nu^+(J\#-J') = \nu^+(-J\#J') = 0.\]
By \cite[Proposition 4.7]{Ozsvath14}, $\nu^+$-equivalent knots have the same $\Upsilon$ function.

\begin{proposition}\label{prop:genwhiteheaddoubles}Let $D$ denote the positive Whitehead double of the right handed trefoil $T_{2,3}$. Then for $\frac{2}{5} \leq t \leq \frac{2}{3}$,
\[\Upsilon_{{(D\# D)}_{3,11}\#-T_{3,11}}(t)= -4+5t
\]
\end{proposition}

\begin{proof}We calculate the $\Upsilon$ invariants of the individual connected summands and the result will follow since $\Upsilon$ is a group homomorphism from $\mathcal{C}^{\Diff}\to \Z$ to $\R$--valued continuous functions on $[0,2]$.

For an $L$--space knot $J$, the $\Upsilon$ invariant can be calculated as follows. Write the \emph{symmetrised} Alexander polynomial of $J$ as $\sum_{k=0}^n(-1)^kt^{\alpha_k}$ where $\{\alpha_k\}_{k=0}^n$ is a decreasing sequence of integers. Define another sequence of integers $\{m_k\}^n_{k=0}$ by~$m_0:=0$, $m_{2i+1}:=m_{2i}-1$ and $m_{2i+2}:=m_{2i+1}+1+2(\alpha_{2i+2}-\alpha_{2i+1})$, for~$i\geq 0$. By \cite[Theorem 6.2]{Ozsvath14}, the $\Upsilon$ invariant of $K$ is then given by\[\Upsilon_J(t)=\max_{\{i\mid0\leq 2i\leq n\}} \{m_{2i}-t\alpha_{2i}\}.\]

\begin{claim}
\[\Upsilon_{{T}_{3,11}}(t)=      -10t \qquad \mbox{\text{where $0\leq t \leq \frac{2}{3}$.}}
\]
\end{claim}

To show this, we begin by using Equation (\ref{eq:1}) to calculate:\[\Delta_{T_{3,11}}= 1-t+t^3-t^4+t^6-t^7+t^9-t^{10}+t^{11}-t^{13}+t^{14}-t^{16}+t^{17}-t^{19}+t^{20}.\]Symmetrise this polynomial by multiplying with $t^{-10}$. With this symmetrised polynomial, and the above method of \cite[Theorem 6.2]{Ozsvath14}, the remainder of the calculation is a straightforward but lengthy optimisation problem. In particular, on the interval~$[0,\frac{2}{3}]$ this calculation results in the $\Upsilon$ invariant $-10t$, completing the proof of the claim.

\begin{claim}
\[\Upsilon_{{(D\# D)}_{3,11}}(t)= -4-5t \qquad \mbox{\text{where $\frac{2}{5}\leq t \leq 1$.}}
\]
\end{claim}
To show this, we begin by citing the fact that the `infinity' version of the Heegaard-Floer knot chain complex (see \cite{MR2065507,MR2704683}) of $D\# D$ is filtered chain homotopy equivalent to that of $T_{2,5}$ up to an acyclic complex \cite[Proposition 6.1]{MR3466802}. As observed in~\cite[Example 2.3]{Kim:2016uq} this implies that $D\# D$ and $T_{2,5}$ are~$\nu^+$--equivalent. Now~\cite[Theorem B]{Kim:2016uq} states that if $P$ is a pattern in the solid torus with nonzero winding number then taking satellite knots using the pattern~$P$ preserves $\nu^+$--equivalence. Taking the $(3,11)$ cable of knot in $S^3$ corresponds to such a satellite operation. So~$(D\# D)_{3,11}$ and $(T_{2,5})_{3,11}$ are $\nu^+$--equivalent. In particular they have the same~$\Upsilon$ invariant. It remains to calculate the $\Upsilon$ invariant of~$(T_{2,5})_{3,11}$.

In \cite[Theorem 1.10]{MR2511910} it is shown that if $J\subset S^3$ is an $L$--space knot and~$s\geq r(2g_3(J)-1)$ then the $(r,s)$ cable $J_{r,s}$ is an $L$--space knot. As~$T_{2,5}$ is an~$L$--space knot, and $g_3(T_{2,5})=2$, so $(T_{2,5})_{3,11}$ is an $L$--space knot as~$11>3(2\cdot2-1)$. Combining Equation (\ref{eq:1}) with the formula in \cite[Theorem II]{Seifert:1950-1} for the Alexander polynomial of a satellite knot, we calculate
\begin{align*}
&\Delta_{(T_{2,5})_{3,11}}(t)\\
= \,&\Delta_{T_{2,5}}(t^3)\cdot\Delta_{T_{3,11}}(t)\\
= \,&1-t+t^6-t^7+t^{11}-t^{-13}+t^{15}-t^{16}+t^{16}-t^{19}+t^{21}-t^{25}+t^{26}-t^{31}+t^{32}.
\end{align*}
Symmetrise this polynomial by multiplying with $t^{-16}$. The remainder of the calculation is another straightforward but lengthy optimisation problem. On the interval~$[\frac{2}{5},1]$ this calculation results in the $\Upsilon$ invariant $-4-5t$, completing the proof of the claim, and of the proposition.
\end{proof}

\section{Topological almost concordance in lens spaces}

We will now restrict ourselves entirely to the topological category and adapt the ideas so far developed in the paper in order to prove the following theorem, which confirms that \cite[Conjecture 1.3]{FNOP} holds for any free homotopy class in any lens space.

\begin{theoremC}
  Let $Y = L(p,q)$, for $\gcd(p,q)=1$, $p>1$, and let $x \in [S^1, Y]$ be any free homotopy class. Then there are infinitely many topological almost concordance classes in $\CC_x^{\Top}(Y)$.
\end{theoremC}

Before we prove the theorem, we recall some results about Levine-Tristram signatures of torus knots. Let $J$ be a knot in $S^3$, and let $A$ be any choice of integral Seifert matrix for $J$. For any $\omega\in\C$ with $|\omega|=1$, the \emph{Levine-Tristram signature} is defined to be \[\sigma_\omega(J):=\sgn\left((1-\omega)A+(1-\overline{\omega})A^T\right)\in\Z,\]where the bar denotes complex conjugation, the $T$ denotes matrix transposition, and `$\sgn$' denotes taking the signature of a hermitian matrix. The value $\sigma_\omega(J)$ is independent of the choice of Seifert matrix $A$. The function $\sigma_\omega(J)\colon S^1\to \Z$ is continuous (and therefore constant) away from the set of roots of the Alexander polynomial $\Delta_J(t)$, where it may `jump' by an even integer. In particular this implies that $\sigma_\omega(J)\in2\Z$ when $\Delta_J(\omega)\neq 0$. For $\omega\in S^1$ and $t\in(-\varepsilon,\varepsilon)$, write~$\omega_t=\omega\cdot\exp(2\pi i t)$ and define the \emph{jump function} to be\[j_J\colon S^1\to2\Z;\qquad j_J(\omega):=\lim_{t\to0^+}\left(\sigma_{\omega_t}(J)-\sigma_{\omega_{-t}}(J)\right).\] Writing $S_J$ for the unit circle with the roots of $\Delta_J(t)$ removed, there is equality~$\sigma_\omega(J\# J')= \sigma_\omega(J)+\sigma_\omega(J')$ for~$\omega\in S_J\cap S_J'$; moreover the signature~$\sigma_\omega(J)$ of a knot $J$ vanishes on $S_J$ if $J$ is topologically concordant to the unknot. See~\cite[\textsection 25]{MR0246314} for proof of these last two facts.

Recalling the Alexander polynomial of the $(r,s)$ torus knot from Equation (\ref{eq:1}), the set of possible jump points for the signature function $\sigma_\omega(T_{r,s})$ is the set of~$(rs)$th roots of unity which are neither $r$th roots of unity, nor $s$th roots of unity. For~$0<N<rs-1$, write $j_{r,s}(N):=j_{T_{r,s}}(\exp(2\pi i N/rs))$ and \[L(N):=\{(i,j)\mid ir+js=N,\, 0\leq i\leq s,\,0\leq j\leq r\}.\] We will make use of the following theorem, originally due to Litherland \cite{Litherland79}, but in a form derived by Collins in \cite[\textsection 3]{Collins10}.

\begin{theorem}[Litherland]\label{thm:litherland}For $0<N<rs-1$, such that $N$ is neither a multiple of $r$, nor of $s$, the jump function of the $(r,s)$ torus knot is\[j_{r,s}\left(N\right)=\left\{\begin{array}{ll}+2&\text{if $|L(N)|=1$,} \\ -2&\text{if $|L(N)|=0$,}\end{array}\right.\]and these are the only possibilities.
\end{theorem}

\noindent We now have enough background to prove Theorem~\ref{thm:C}.

\begin{proof}[Proof of Theorem~\ref{thm:C}]
The case of $x$ the null homotopic class was already proved in~\cite[Theorem 1.5]{FNOP}. We shall therefore proceed by assuming that $x$ is not null homotopic. We choose an oriented meridian to the $-p/q$ surgery curve in Figure~\ref{fig:CoveredLink}, which determines an isomorphism $\pi_1(L(p,q))\cong\Z/p\Z$, under which identification we have~$[x]=:\ell$, where $0<\ell<p$.  For $n\geq0$, let $K_{n,\ell}$ be the link in $Y$ that is described in Figure~\ref{fig:downstairs}, with the orientation of the strands agreeing with the chosen orientation on the meridian of the surgery curve.

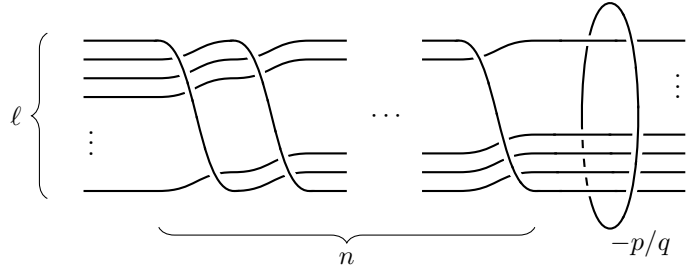
\begin{figure}[h]

\begin{tikzpicture}\begin{knot}[
clip width=5,
flip crossing=1,
flip crossing=2,
flip crossing=4,
flip crossing=6,
flip crossing=7,
flip crossing=5,
flip crossing=3,
flip crossing=18,
flip crossing=15,
flip crossing=12,
flip crossing=9,
flip crossing=19,
flip crossing=16,
flip crossing=13,
flip crossing=10,
flip crossing=21,
]

\strand[black,thick] (-1,0) ..  controls +(0,0) and +(-0,0) .. (0,0) .. controls +(0.5,0) and +(-0.5,0) .. (1,0.25) .. controls +(0.5,0) and +(-0.5,0) .. (2,0.5) .. controls +(0,0) and +(-0,0) .. (2.5,0.5);

\strand[black,thick] (-1,1.25) ..  controls +(0,0) and +(-0,0) .. (0,1.25) .. controls +(0.5,0) and +(-0.5,0) .. (1,1.5) .. controls +(0.5,0) and +(-0.5,0) .. (2,1.75) .. controls +(0,0) and +(-0,0) .. (2.5,1.75);

\strand[black,thick] (-1,1.5) ..  controls +(0,0) and +(-0,0) ..  (0,1.5) .. controls +(0.5,0) and +(-0.5,0) .. (1,1.75) .. controls +(0.5,0) and +(-0.5,0) .. (2,2)  .. controls +(0,0) and +(-0,0) .. (2.5,2);

\strand[black,thick] (-1,1.75) ..  controls +(0,0) and +(-0,0) ..  (0,1.75) .. controls +(0.5,0) and +(-0.5,0) .. (1,2) .. controls +(0.5,0) and +(-0.5,0) .. (2,0)  .. controls +(0,0) and +(-0,0) .. (2.5,0);

\strand[black,thick] (-1,2) ..  controls +(0,0) and +(-0,0) ..  (0,2) .. controls +(0.5,0) and +(-0.5,0) .. (1,0) .. controls +(0.5,0) and +(-0.5,0) .. (2,0.25) .. controls +(0,0) and +(-0,0) .. (2.5,0.25);

\strand[black,thick] (3.5,0) ..  controls +(0,0) and +(-0,0) ..  (4,0) .. controls +(0.5,0) and +(-0.5,0) .. (5,0.25) .. controls +(0.5,0) and +(-0.5,0) .. (7,0.25);

\strand[black,thick] (3.5,0.25) ..  controls +(0,0) and +(-0,0) .. (4,0.25) .. controls +(0.5,0) and +(-0.5,0) .. (5,0.5) .. controls +(0.5,0) and +(-0.5,0) .. (7,0.5);

\strand[black,thick] (3.5,0.5) ..  controls +(0,0) and +(-0,0) .. (4,0.5) .. controls +(0.5,0) and +(-0.5,0) .. (5,0.75) .. controls +(0.5,0) and +(-0.5,0) .. (7,0.75);

\strand[black,thick] (3.5,1.75) ..  controls +(0,0) and +(-0,0) .. (4,1.75) .. controls +(0.5,0) and +(-0.5,0) .. (5,2) .. controls +(0.5,0) and +(-0.5,0) .. (7,2);

\strand[black,thick] (3.5,2) ..  controls +(0,0) and +(-0,0) .. (4,2) .. controls +(0.5,0) and +(-0.5,0) .. (5,0) .. controls +(0.5,0) and +(-0.5,0) .. (7,0);

\strand[black,thick] (6,2.5) .. controls +(0.5,0) and +(0.5,0) .. (6,-0.5) .. controls +(-0.5,0) and +(-0.5,0) .. (6,2.5);

\end{knot}

\node at (6.4,-0.7) {$-p/q$};
\node at (3.05,1) {$\cdots$};
\node at (-0.9,0.75) {$\vdots$};
\node at (6.9,1.5) {$\vdots$};

\draw [decorate,decoration={brace,amplitude=6pt,raise=8pt}] (-1.2,-0.1) -- (-1.2,2.1);
\node at (-1.9,1) {$\ell$};

\draw [decorate,decoration={brace,amplitude=6pt,raise=8pt}]  (5,-0.2) -- (0,-0.2);
\node at (2.5,-0.9) {$n$};

\end{tikzpicture}
\caption{The link $K_{n,\ell}$ in $L(p,q)$, which is a copy of the torus link $T_{\ell,n}$ drawn in a neighbourhood of the standard meridian to the $-p/q$ surgery curve.}
\label{fig:downstairs}
\end{figure}

When $K_{n,\ell}$ is moreover a knot, it clearly represents the free homotopy class $x$, by construction. The link $K_{n,\ell}$ is a knot if and only if the number of meridional twists~$\ell$ is coprime to the number of longitudinal twists~$n$. The set \[\mathcal{A}:=\{n\mid\gcd(\ell,n)=1,\,n>0\}\] is infinite, and we will show that the knots in the infinite family $\{K_{n,\ell}\}_{n\in \mathcal{A}}$ belong to mutually distinct topological almost concordance classes.

Let $\wt{L}_{n,\ell}$ in $S^3$ be the covering link of $K_{n,\ell}$. Then by Lemma~\ref{lem:coveringlink-lens-spaces}, $\wt{L}_{n,\ell}$ is the torus link $T_{\ell,pn}$. Writing $d:=\gcd(\ell,p)$, note that when $\gcd(\ell,n)=1$, we have~$\gcd(\ell,pn)=d$ and so $\wt{L}_{n,\ell}$ is a $d$--component link. For each $n\in \mathcal{A}$, fix a preferred component $\wt K_{n,\ell}$ of the covering link $\wt{L}_{n,\ell}$ and note that $\wt K_{n,\ell}$ is the $\left(\frac{\ell}{d}\, ,\frac{pn+\ell q}{d}\right)$ torus knot.

Fix any $n_1, n_2\in \mathcal{A}$, with $n_1 < n_2$. Suppose, for a contradiction, that $K_{n_1,\ell}$ is topologically almost concordant to $K_{n_2,\ell}$. Write $\sim$ to indicate the relation of topological concordance. As the degree of the covering from $\wt{K}_{n_1,\ell}$ to $K_{n,\ell}$ is $p/d$, the discussion in Section~\ref{subsec:coveringlinks} implies there exists a knot $J$ in $S^3$ such that
\[\wt K_{n_1,\ell}  \#  \left(\frac{p}{d}\right) J \sim \wt K_{n_2,\ell}.\]Equivalently, using a minus sign to denote taking the mirror image and reversing the orientation, the knot $K:= \wt K_{n_1,\ell}  \#  \left(\frac{p}{d}\right) J \# -\wt K_{n_2,\ell}$ is concordant to the unknot. Set $S$ to be the unit circle $S^1\subset \C$ but with the roots of the Alexander polynomials of $\wt K_{n_1,\ell}$, $\wt K_{n_2,\ell}$, and of $J$ removed. Combining several properties of Levine-Tristram signature which we have discussed, for all $\omega\in S$ we obtain an equality
 \begin{equation}\label{eq:2}\sigma_{\omega}(\wt K_{n_1,\ell}) + \frac{p}{d}\cdot\sigma_{\omega}(J) = \sigma_{\omega}(\wt K_{n_2,\ell}).\end{equation}
 We will need the following claim.

\begin{claim} For $n\in \mathcal{A}$, the least $t>0$ such that the Levine-Tristram signature of $\wt{K}_{n,\ell}$ has a non-zero jump at $\omega=\exp(2\pi i t)$ is $t=\frac{d^2}{\ell(pn+\ell q)}$. Moreover, this jump is $+2$ when $\ell$ divides $p$, and $-2$ otherwise.
\end{claim}

To see the claim, consider that by Theorem~\ref{thm:litherland} non-zero jumps occur exactly at the points $\exp(2\pi i t)$ when $t=\frac{N}{rs}$, with $N,r,s$ as in Theorem~\ref{thm:litherland}. In our case $(r,s)=\left(\frac{\ell}{d}\, ,\frac{pn+\ell q}{d}\right)$. The least~$t>0$ returning a jump point happens when $N=1$, implying $t=\frac{1}{rs}=\frac{d^2}{\ell(pn+\ell q)}$ at this point, as claimed. To show the second part of the claim, suppose\[i\cdot\left(\frac{\ell}{d}\right)+j\cdot \left(\frac{pn+\ell q}{d}\right)=1\] for $0\leq i\leq \frac{pn+\ell q}{d}$ and $0\leq j\leq \frac{\ell }{d}$. Then either $(i,j)=(0,1)$ or $(i,j)=(1,0)$ as~$i, j\geq 0$ and $\ell/d$, $(pn+\ell q)/d$ are integers. In the first case this implies $pn+\ell q=d$, which leads to a contradiction as $p,q>1$. In the second case this implies $\ell =d$, so that $\ell $ divides $p$. Hence $|L(1)|=1$ if and only if $\ell $ divides $p$, and $|L(1)|=0$ otherwise. By Theorem~\ref{thm:litherland}, the claim is proved.

Recall that $n_1<n_2$, so the claim implies the first jump point occurs at a smaller~$t$--value for $\wt{K}_{n_2,\ell}$ than for $\wt{K}_{n_1,\ell}$. With this in mind, choose $\omega_0=\exp(2\pi it_0)\in S$ such that \[\frac{d^2}{\ell(pn_2+\ell q)} < t_0 <\frac{d^2}{\ell(pn_2+\ell q)}+\varepsilon\] for small $\varepsilon >0$.  Choose $\varepsilon$ so that at $t_0$ the first jump has occurred for $\wt{K}_{n_2,\ell}$, but not for $\wt{K}_{n_1,\ell}$. Such a choice is possible as $S$ is dense in $S^1$. Inputting this information to Equation (\ref{eq:2}), we obtain that $0+\frac{p}{d}\cdot \sigma_{\omega_0}(J)=\pm2$. Recalling that $\sigma_{\omega_0}(J)\in2\Z$, this implies that $p=\pm d$. But as $d=\gcd(\ell ,p)>0$, this implies that $p$ divides $\ell $, which is a contradiction as $\ell <p$.
\end{proof}

\bibliographystyle{alpha}
\bibliography{biblio_smoothalmost}

\end{document}